\documentclass[reqno,11pt]{amsart}
\usepackage{amsmath,amsfonts,amssymb,amsthm}
\usepackage{mathrsfs}
\usepackage{xcolor}

 \allowdisplaybreaks
\numberwithin{equation}{section}

\def\be{\begin{equation}}
\def\ee{\end{equation}}
\def\eps{\varepsilon}
\def\R{{\mathbb R}}
\def\C{{\mathbb C}}
\newcommand{\wh}{\widehat}
\renewcommand{\r}{\rho}
\newcommand{\lf}{\left}
\newcommand{\ri}{\right}
\def\pt{\partial}
\newcommand*\di{\mathop{}\!\mathrm{d}}
\def\bal{\begin{aligned}}
\def\eal{\end{aligned}}
\def\proofof#1{\begin{proof}[Proof of #1]}
\def\step#1#2{

\noindent{\underline{\it Step~#1.}}\emph{ #2}\\}
\usepackage{mathtools}
\mathtoolsset{showonlyrefs}
\newcommand{\cL}{\mathcal L}
\newcommand{\cT}{\mathcal T}
\DeclareMathOperator{\re}{Re}
\DeclareMathOperator{\spt}{supp}
\DeclareMathOperator{\sgn}{sgn}

\theoremstyle{plain}
\newtheorem{lemma}{Lemma}[section]
\newtheorem{prop}[lemma]{Proposition}
\newtheorem{theorem}[lemma]{Theorem}
\newtheorem{corol}[lemma]{Corollary}
\newtheorem{defin}[lemma]{Definition}
\newtheorem{remark}[lemma]{Remark}

\begin{document}

\title[Non-uniqueness for forced 1D Vlasov-Poisson]{Non-uniqueness for the forced\\ one-dimensional Vlasov-Poisson equation}

\author{Mikaela Iacobelli}
\thanks{ETH Z\"urich, Department of Mathematics, R\"amistrasse 101, 8092 Z\"urich, Switzerland.\newline Email: \textsf{mikaela.iacobelli@math.ethz.ch}}

\begin{abstract}
We prove non-uniqueness for the forced one-dimensional Vlasov-Poisson equation on the whole space. More precisely, for every finite $p$, we construct a nonnegative ion density, smooth for positive times, and a source $S\in C([0,1];L^1(\R^2))$ with $S(0)=0$, for which there exist two distinct nonnegative distributional solutions $f_\pm\in C([0,1];L^1(\R^2))$ with the same zero initial datum and with spatial densities bounded in $L^p(\R)$ uniformly in time. The result is sharp in terms of the integrability of the spatial densities: our solutions have Lipschitz electric fields, and we show that in this class uniqueness holds as soon as the spatial density is bounded.
\end{abstract}

\maketitle

\begin{center}
\textit{{In memory of Claude Bardos, with gratitude and admiration.}}
\end{center}

\vspace{0.5cm}

\section{Introduction}

\subsection{General overview}
The Vlasov-Poisson system is the classical kinetic model for a dilute, fully ionized, unmagnetized plasma, in which $f=f(t,x,v)$ describes the distribution of the electrons moving in their self-consistent electrostatic field, while the ions form a neutralizing background.

The well-posedness of this model has been studied extensively over the past decades, starting with the case of Vlasov equations with globally Lipschitz interaction forces. In that case, Dobrushin \cite{Dob} proved the existence, uniqueness, and stability of measure-valued solutions with finite first moment, and we refer to \cite{Golse16,Jabin14} for a discussion of this theory and of mean-field limits for Vlasov-type equations.

For the Vlasov-Poisson system, where the interaction force is not Lipschitz, global weak solutions on $\R^3$ were constructed in \cite{Arsenev} for nonnegative initial data in $L^1\cap L^\infty$ with finite energy. Under suitable assumptions on the initial data, global classical solutions on the whole space were obtained in one, two and three dimensions in \cite{Iordanskii,UkaiOkabe,Pfaffelmoser}, respectively.
In three dimensions, global classical solutions for sufficiently small and localized initial data were obtained earlier in \cite{BardosDegond85}, while
a different approach to global existence in three dimensions, based on the propagation of velocity moments, was developed in \cite{LionsPerthame}. Global classical solutions on the three-dimensional torus were then constructed in \cite{BattRein}.
We refer to \cite{GPIwp} for a more detailed account of the well-posedness theory for Vlasov-type equations.

Questions of uniqueness and stability for the Vlasov-Poisson system have also received much attention.
In his fundamental work \cite{Loep}, Loeper proved uniqueness for solutions with bounded spatial density by estimating the $2$-Wasserstein distance between two such solutions, using the relation between $W_2$ and the $\dot H^{-1}$ norm, and the same approach also gives a new proof of the Yudovich uniqueness result for the two-dimensional Euler equations \cite{Yud}. The bounded-density criterion was extended in \cite{MiotCMP2016} to spatial densities whose $L^p$ norms grow at most linearly in $p$, uniformly in time, while uniqueness and stability for densities in suitable Orlicz spaces were proved in \cite{HoldingMiot2018}.
Also, in one dimension, a weak-strong stability estimate in the $1$-Wasserstein distance was proved in \cite{Hau13} (we shall return to weak-strong uniqueness for the forced equation in Appendix~\ref{sec:uniqueness}).
The kinetic-Wasserstein distances introduced in \cite{Iac22}, which reflect the natural anisotropy between the position and velocity variables, also improve the classical stability estimates of Dobrushin \cite{Dob} and Loeper.
The extension to $p$-kinetic Wasserstein distances, for $1<p<\infty$, was obtained in \cite{IJ24}, and an application of this approach to the semiclassical limit from Hartree to Vlasov-Poisson is given in \cite{IL24}.
These uniqueness results require bounded spatial densities or a quantitative Osgood-type control of the growth of their $L^p$ norms, while uniqueness under weaker density assumptions remains largely open.

For more singular solutions, non-unique weak solutions involving charge concentrations were constructed for the one-component one-dimensional Vlasov-Poisson equation in \cite{MMZ94a}, while screening and the need for measure-valued solutions in the two-component problem were studied in \cite{MMZ94b}. These examples are different from the present result, which produces $L^1$-valued solutions that are continuous up to time zero, and have zero initial datum.

We also refer to \cite{Winter23} for a non-uniqueness result for stationary solutions in the presence of attractive background charges.

Several non-uniqueness constructions have recently been developed for equations of fluid dynamics.
The idea that non-uniqueness could arise from a self-similar instability was proposed for the unforced Navier-Stokes equations by Jia and \v{S}ver\'ak, who constructed forward self-similar solutions for large, scale-invariant data \cite{JS14} and proved a conditional non-uniqueness result under a spectral assumption on the linearized operator \cite{JS15}. Numerical support for this scenario was provided by Guillod and \v{S}ver\'ak \cite{GS23}. For every $2<p<\infty$, Vishik \cite{Vishik18a,Vishik18b} constructed two distinct weak solutions of the forced two-dimensional Euler equations with the same initial vorticity in $L^1\cap L^p$ and the same force, and a self-contained treatment, including a zero-initial-vorticity variant, is given in \cite{ABCDLGJK}. Albritton, Bru\'e and Colombo \cite{ABC22} constructed two distinct Leray solutions of the forced Navier-Stokes equations with zero initial velocity and identical body force. The same strategy has subsequently been applied to several other equations of fluid dynamics, including further Navier-Stokes settings \cite{ABCglue,AC23}, the forced two-dimensional Euler equations \cite{CFMS-Euler,DM25}, and the forced surface quasi-geostrophic equation \cite{CFMS-SQG}.

To explain the idea behind these results, we introduce the self-similar time $\tau=\log t$ and observe that, if the equation linearized in similarity variables around a stationary profile has an unstable eigenvalue $\lambda$ with $\re\lambda>0$, then the unstable mode $e^{\lambda\tau}\phi=t^\lambda\phi$ vanishes as $t\to0^+$.
One then constructs nonlinear solutions from this mode and checks that its decay compensates for the rescaling, so that the solutions have the same initial trace in physical variables. An additive source can make the chosen background stationary in similarity variables without affecting the linearization, so one first needs an unstable eigenvalue with a localized eigenfunction.

For Vlasov-Poisson, our choice of the two-stream instability and its scale is inspired by our work on quasineutral limits (see \cite[Appendix~B]{GPIstab}).
The classical unstable modes around a spatially homogeneous equilibrium are plane waves, so neither they nor the background are integrable over the whole space. We construct a localized unstable mode by reducing the kinetic eigenvalue problem to a compact operator acting on the electric field.
Concentrating the background on a scale $\eps$ and choosing an amplitude of order $\eps^{-2}$, as in the Debye scaling, produces a growth rate of order $\eps^{-1}$ and allows us to retain the instability when restoring the drift generated by the self-similar change of position and velocity variables.

Unlike the vorticity in the Euler equation, the kinetic distribution must remain nonnegative, so to construct the two solution branches we need to bound the eigenfunction pointwise by a multiple of a positive integrable background.
We obtain this bound by adding a positive tail adapted to the kinetic characteristics and using the large growth rate to ensure that the characteristic integral in the pointwise estimate converges.

{The two-stream instability is a particular Penrose instability associated with velocity profiles having two sufficiently separated peaks \cite{Penrose60}.} Nonlinear instability of Penrose-unstable double-humped equilibria was proved in \cite{GuoStraussDouble}, while spatially inhomogeneous periodic BGK equilibria, BGK solitary waves, and collisionless shocks in relativistic plasma models were studied in \cite{GuoStraussPeriodic,GuoStraussSolitary}. See \cite{Baradat18,PankavichAllen,HKN16} for further instability results in the Vlasov setting { (see also \cite{BardosNouri12,BardosBesse13} for the related Vlasov-Dirac-Benney equation)}.
{In the physics literature, beam-plasma oscillations and current-carrying plasmas were studied in \cite{BohmGross49,Buneman58}, respectively, and we also refer to \cite{ONeil65} for particle trapping, to \cite{RobertsBerk67} for the nonlinear evolution of the two-stream instability, and to \cite{BGK57} for the construction of BGK waves.}

In the quasineutral scaling, $\eps$ represents the ratio between the Debye length and the observation scale, and structures at the Debye scale evolve at a plasma frequency of order $\eps^{-1}$ in macroscopic time. The limit for data with uniform analytic regularity was established in \cite{Grenier96}, while the two-stream instability and the associated fast scale were identified in \cite{Grenier99} as an obstruction to quasineutral convergence. A rigorous failure result for Penrose-unstable homogeneous velocity profiles was obtained in \cite{HKH15}, while stability under the sharp Penrose condition was proved in \cite{HKR16}. In \cite{HKI1,HKI2,Iac22}, this dichotomy was quantified through Wasserstein-type stability estimates, with admissible perturbations exponentially small in $\eps$.
Analogous results for the ionic model were obtained in \cite{GPIJMPA,GPIstab}; see \cite{Geom} for a recent overview.

\subsection{The main result}
Throughout the paper, for a phase-space function $g$ we write $\rho_g(x):=\int_\R g(x,v)\,\di v$ for its spatial density, using the same notation in the rescaled variables. In one space dimension, the forced Vlasov-Poisson equation reads
\be\label{eq:vp}
 \lf\{
 \begin{array}{l}
 \pt_t f+v\,\pt_x f+E\,\pt_v f=S\\[1mm]
 \pt_xE=\r_f-\r_i
 \end{array}
 \ri.,
\ee
where $\r_i=\r_i(t,x)$ is a prescribed, possibly time-dependent, nonnegative ion density and $S$ is an additive kinetic source. We take $E=E[\r_f-\r_i]$, where for $q\in L^1(\R)$ we define
\be\label{eq:field}
 E[q](x):=\frac12\int_\R \sgn(x-y)\,q(y)\,\di y,
 \qquad \pt_xE[q]=q.
\ee
Solutions of \eqref{eq:vp} are understood in the sense of distributions.

Our main result is the following.

\begin{theorem}[Forced non-uniqueness from zero initial data]
\label{thm:main}
For every finite $p\in[1,\infty)$ there exist a prescribed nonnegative ion density
\[
 \r_i\in C^\infty((0,1]\times\R),
\]
a source
\[
 S\in C([0,1];L^1(\R^2)),
 \qquad S(0)=0,
\]
and two distinct nonnegative distributional solutions $f_+,f_-$ of \eqref{eq:vp} such that
\be\label{eq:initial}
 f_\pm\in C([0,1];L^1(\R^2)),
 \qquad f_+(0)=f_-(0)=0,\qquad \r_{f_\pm}\in L^\infty((0,1);L^p(\R)).
\ee
Also, for every $t\in(0,1]$, the solutions have a finite first velocity moment and satisfy
\[
 \int_\R\bigl(\r_{f_\pm}(t,x)-\r_i(t,x)\bigr)\,\di x=0.
\]
\end{theorem}

{
For a common source $S\in L^1((0,T)\times\R^2)$ and a prescribed nonnegative ion density $\rho_i\in L^1((0,T)\times\R)$, Theorem~\ref{thm:uniqueness} gives uniqueness among nonnegative distributional solutions with the same initial datum, continuous in $L^1(\R^2,(1+|x|+|v|)\,\di x\,\di v)$, provided one solution satisfies
\[
 E_1\in L^1(0,T;W^{1,\infty}(\R)),
 \qquad \int_0^T t\|\rho_{f_1}(t)\|_{L^\infty}\,\di t<\infty.
\]
The factor $t$ comes from the second-order structure of the kinetic characteristics. In particular, this criterion allows densities which are unbounded as $t\to0^+$.

The solutions of Theorem~\ref{thm:main} satisfy the weighted continuity and field regularity assumptions above, while Proposition~\ref{prop:endpoint} shows that, for some $m>0$,
\[
 \lim_{t\to0^+}t^2\|\rho_{f_\pm}(t)\|_{L^\infty}=m,
 \qquad t^2\|\rho_i(t)\|_{L^\infty}=m.
\]
The time weight is therefore sharp, since the density integral diverges logarithmically, but becomes finite if $t$ is replaced by $t^{1+\eta}$ for any $\eta>0$.
The same examples show that weak $L^1$ cannot replace $L^1$ in the condition on $t\|\rho_{f_1}(t)\|_{L^\infty}$ without additional assumptions.
{
If $E_1=E[\rho_{f_1}-\rho_i]\in L^1(0,T;W^{1,\infty}(\R))$, then by Corollary~\ref{cor:ions} we have
\[
 \int_0^T t\|\rho_{f_1}(t)\|_{L^\infty}\,\di t<\infty
 \quad\Longleftrightarrow\quad
 \int_0^T t\|\rho_i(t)\|_{L^\infty}\,\di t<\infty.
\]
}
For the ion density of our examples this integral diverges, so the optimality concerns a class in which the prescribed ions may vary, rather than a transition between uniqueness and non-uniqueness for one fixed background. Within each non-unique pair, both the ions and the source are identical.

}

{The passage from the unstable mode to two solutions with the same source relies on the innovative use of Golovkin's identity \cite{Gol} by Dolce and Mescolini \cite{DM25}.}
{In suitable self-similar variables}, we construct a positive background and a localized unstable mode bounded pointwise by a multiple of that background, so adding and subtracting a sufficiently small multiple of the growing mode gives two nonnegative functions.
Since the mode satisfies the linearized equation, the linear terms cancel against the time derivative, while the quadratic term is unchanged when the sign of the perturbation is reversed, and we can therefore choose one source for which both functions solve the equation.

\smallskip

The paper is structured as follows:
Section~\ref{sec:similarity} introduces the self-similar equation and the spatial rescaling, and Section~\ref{sec:localized} constructs the localized unstable mode. We restore the similarity drift in Section~\ref{sec:drift} and add a positive tail to the background in Section~\ref{sec:tails}, before proving Theorem~\ref{thm:main} in Section~\ref{sec:proof}.
Appendix~\ref{sec:uniqueness} contains the weak-strong uniqueness criterion and the proof of its temporal sharpness.

 \subsection{AI declaration} The mathematical ideas and
proofs in this paper were developed by the author.
AI tools have been used to improve the presentation and exposition.

\section{Self-similar formulation and spatial concentration}
\label{sec:similarity}

We first derive the equation in self-similar variables and its linearization, then rescale a spatially concentrated background.

\subsection{The self-similar equation}

We rescale position and velocity by setting
\be\label{eq:similarity}
 f(t,x,v)=t^aF(\tau,\xi,w),
 \qquad
 \tau=\log t,
 \qquad
 \xi=x\,t^{-b},
 \qquad
 w=v\,t^{-c}.
\ee
We fix a sufficiently regular nonnegative profile \(\overline F=\overline F(\xi,w)\), independent of \(\tau\), and choose the prescribed ion density as
\be\label{eq:ions}
 \r_i(t,x)=t^{a+c}\rho_{\overline F}(\xi).
\ee
By \eqref{eq:similarity}, we have \(\r_f=t^{a+c}\r_F\), so \eqref{eq:field} gives
\[
E(t,x)=t^{a+b+c}E[\rho_F(\tau,\cdot)-\rho_{\overline F}](\xi).
\]
Applying the chain rule, we obtain
\begin{align*}
 \pt_t f
 &=t^{a-1}\bigl(\pt_\tau F-b\xi\pt_\xi F-cw\pt_wF+aF\bigr),\\
 v\pt_xf&=t^{a+c-b}w\pt_\xi F,\\
 E\pt_vf&=t^{2a+b}E[\rho_F-\rho_{\overline F}]\pt_wF,
\end{align*}
so, balancing the three powers of \(t\), we find
\[
a-1=a+c-b=2a+b,
\]
and hence
\be\label{eq:scaling}
 b=c+1,
 \qquad
 a=-2-c,
\ee
which leaves $b$ free; since $a+c=-2$, the spatial density satisfies, for $1\leq p<\infty$,
\be\label{eq:rho}
 \|\rho_f(t)\|_{L^p}
 =t^{-2+b/p}\|\rho_F(\log t)\|_{L^p}.
\ee
Thus, if the rescaled density stays bounded in $L^p$ as $\tau\to-\infty$, choosing $b\geq2p$ keeps the physical density bounded in $L^p$ near $t=0$. The stronger choice $b>\max\{3,2p\}$ in Section~\ref{sec:proof} will also ensure that the source tends to zero in $L^1$.

Under \eqref{eq:similarity} and \eqref{eq:scaling}, the three terms on the left-hand side of the Vlasov-Poisson equation have the common factor $t^{a-1}$, so we define the source in self-similar variables by
\[
 \mathscr S(\tau,\xi,w)
 :=t^{1-a}S(t,t^b\xi,t^cw),
 \qquad t=e^\tau,
\]
and, after dividing the equation by $t^{a-1}$, we can write
\[
 \pt_\tau F=\mathcal A_{\overline F}(F)+\mathscr S,
\]
where
\be\label{eq:AF}
 \mathcal A_{\overline F}(F)
 :=-w\,\pt_\xi F-E[\rho_F-\rho_{\overline F}]\,\pt_wF
   +b\,\xi\,\pt_\xi F+c\,w\,\pt_wF-aF.
\ee
By our choice of the ion density, the electric field vanishes at $F=\overline F$, so
\[
 \mathcal A_{\overline F}(\overline F)
 =-w\,\pt_\xi\overline F+b\,\xi\,\pt_\xi\overline F
   +c\,w\,\pt_w\overline F-a\overline F.
\]

To compute the linearization, we keep the ion density $\rho_{\overline F}$ fixed and write $F=\overline F+h$, so
\[
 \rho_{\overline F+h}-\rho_{\overline F}=\rho_h,
 \qquad
 -E[\rho_h]\,\pt_w(\overline F+h)
 =-E[\rho_h]\,\pt_w\overline F-E[\rho_h]\,\pt_wh.
\]
Since the second term on the right-hand side is quadratic in $h$, the linearized operator around $\overline F$ is
\be\label{eq:linearization}
 \cL_{\overline F}h
 :=D\mathcal A_{\overline F}(\overline F)h
 =-w\,\pt_\xi h+b\,\xi\,\pt_\xi h+c\,w\,\pt_wh-ah
   -E[\rho_h]\,\pt_w\overline F,
\ee
and, with $\mathscr S=-\mathcal A_{\overline F}(\overline F)$, the perturbation satisfies
\[
 \pt_\tau h=\cL_{\overline F}h-E[\rho_h]\,\pt_wh.
\]
\subsection{{Spatial rescaling}}
\label{sec:debye}

The two branches will be obtained by adding and subtracting a solution of the linearized equation, so that the quadratic term is the same for both and can be included in their common source.
A positive eigenvalue gives the time factor $e^{\lambda\tau}=t^\lambda$, so the perturbations decay in the rescaled variables as $t\to0^+$. We therefore want to construct a nonnegative background $\overline F$ and a nonzero eigenfunction $h$ satisfying
\[
 \cL_{\overline F}h=\lambda h,
 \qquad \lambda>0.
\]
We first consider backgrounds of the form
\be\label{eq:background}
 \overline F_\eps(\xi,w)
 =\eps^{-2}\,n(\xi/\eps)\,\mu(w),
\ee
where $n\geq0$ is a spatial profile, $\mu\geq0$ is a velocity distribution with unit mass, and $0<\eps<1$. Here $\eps$ is a spatial concentration parameter, which enters only through the chosen background and which we choose sufficiently small and then fix. The amplitude $\eps^{-2}$ makes the transport and linearized field terms have the same order after rescaling space, a choice motivated by the Debye scaling used in quasineutral limits (see for instance \cite{Grenier96,HKH15,HKR16,HKI2,GPIstab,Geom}), in which a fixed velocity scale and a homogeneous density of order $\eps^{-2}$ correspond to a characteristic length of order $\eps$.
Since $\mu$ has unit mass, the spatial density of the background is
\[
 \rho_{\overline F_\eps}(\xi)
 =\int_\R\overline F_\eps(\xi,w)\,\di w
 =\eps^{-2}n(\xi/\eps)\int_\R\mu(w)\,\di w
 =\eps^{-2}n(\xi/\eps).
\]
As in \eqref{eq:ions}, we prescribe this function as the self-similar ion density, so that the electric field vanishes at $F=\overline F_\eps$.

To work at a fixed spatial scale, we introduce the variable $y=\xi/\eps$ and define the spatial rescaling operator by
\be\label{eq:rescaling}
 (\mathcal R_\eps H)(\xi,w):=H(\xi/\eps,w),
 \qquad h=\mathcal R_\eps H.
\ee
Note that, since the rescaling acts only on the spatial variable, the densities associated with $h$ and $H$ satisfy
\be\label{eq:rhoeps}
 \rho_h(\eps y)=\rho_H(y).
\ee
For a smooth compactly supported $H$, the change of variables gives
\be\label{eq:derivatives}
 \pt_\xi h(\eps y,w)=\eps^{-1}\pt_yH(y,w),
 \qquad \xi\pt_\xi h(\eps y,w)=y\pt_yH(y,w),\qquad
 \pt_wh(\eps y,w)=\pt_wH(y,w).
\ee
Also, by \eqref{eq:field}, the electric field scales as
\be\label{eq:Eeps}
 E[\rho_h](\eps y)=\eps\,E[\rho_H](y),
\ee
while
\be\label{eq:Fbar}
 \pt_w\overline F_\eps(\eps y,w)=\eps^{-2}n(y)\mu'(w),
\ee
so the free transport and the field term are both of order $\eps^{-1}$, whereas the terms generated by the self-similar change of variables,
\[
 b\,y\,\pt_yH+c\,w\,\pt_wH-aH,
\]
remain of order one. We therefore separate these contributions by writing
\be\label{eq:Lloc}
 \cL_{n,\mu}^{\rm loc}H
 :=-w\,\pt_yH-n(y)\,E[\rho_H](y)\,\mu'(w),\qquad
 \mathcal D_{b,c}H
 :=b\,y\,\pt_yH+c\,w\,\pt_wH-aH.
\ee
The eigenfunction will initially be constructed in $L^2$, so we need the reduction to remain valid without assuming that it is smooth. To control its spatial density, we use the mixed norm
\[
 \|H\|_{L^1_wL^2_y}:=\int_\R\|H(\cdot,w)\|_{L^2_y}\,\di w,
\]
so the condition $H\in L^1_wL^2_y$ used in the spectral argument below gives $\rho_H\in L^2(\R)$. For $q\in L^2(\R)$ but not necessarily integrable, we write $E[q]$ for its unique $L^2$ primitive whenever it exists.

\begin{lemma}
\label{lem:debye}
Let $0<\eps<1$.
Assume $n\in C_c^\infty(\R)$, $n\geq0$, and $\mu\in C^\infty(\R)\cap L^1(\R)$, $\mu\geq0$, with
\[
 \int_\R \mu(w)\,\di w=1,
 \qquad \mu'\in L^2(\R).
\]
Let $\overline F_\eps$ be given by \eqref{eq:background} and $\cL_{\overline F_\eps}$ by \eqref{eq:linearization}. Suppose that
\[
 H\in {L^2_{y,w}}\cap L^1_wL^2_y,
 \qquad
 E[\rho_H]\in L^2(\R),
\]
and that the distribution
\[
 \bigl(\cL_{n,\mu}^{\rm loc}+\eps\mathcal D_{b,c}\bigr)H
\]
belongs to ${L^2_{y,w}}$. Then, for $h=\mathcal R_\eps H$ with $\mathcal R_\eps$ as in \eqref{eq:rescaling}, $\cL_{\overline F_\eps}h$ also belongs to ${L^2_{\xi,w}}$ and
\be\label{eq:conjugacy}
 \mathcal R_\eps^{-1}\cL_{\overline F_\eps}\mathcal R_\eps H
 =\eps^{-1}\cL_{n,\mu}^{\rm loc}H+\mathcal D_{b,c}H.
\ee
In particular, for every $\lambda\in\C$, $h=\mathcal R_\eps H$ satisfies $\cL_{\overline F_\eps}h=\lambda h$ if and only if
\be\label{eq:scaled}
 \bigl(\cL_{n,\mu}^{\rm loc}+\eps\,\mathcal D_{b,c}\bigr)H
 =\eps\,\lambda H.
\ee
\end{lemma}

\begin{proof}
The derivative identities \eqref{eq:derivatives} remain valid for distributions by duality.
Moreover, by \eqref{eq:rhoeps},
\[
 \pt_\xi\lf(\eps\,E[\rho_H](\xi/\eps)\ri)
 =\rho_H(\xi/\eps)=\rho_h(\xi),
\]
and since $\eps\,E[\rho_H](\xi/\eps)$ belongs to $L^2(\R)$, uniqueness of the $L^2$ primitive gives
\[
 E[\rho_h](\xi)=\eps\,E[\rho_H](\xi/\eps).
\]
Substituting this field identity and \eqref{eq:derivatives} into \eqref{eq:linearization}, and using \eqref{eq:Fbar}, we obtain
\[
 \cL_{\overline F_\eps}\mathcal R_\eps H
 =\eps^{-1}\mathcal R_\eps
 \bigl(\cL_{n,\mu}^{\rm loc}+\eps\mathcal D_{b,c}\bigr)H.
\]
Since $\mathcal R_\eps$ maps $L^2_{y,w}$ into $L^2_{\xi,w}$ with norm $\eps^{1/2}$, the assumption in the statement ensures that the right-hand side belongs to ${L^2_{\xi,w}}$, so applying $\mathcal R_\eps^{-1}$ gives \eqref{eq:conjugacy}, and the equivalence of the two eigenvalue problems follows by multiplying the conjugated equation by $\eps$.
\end{proof}

\begin{remark}
In \cite{HKH15,GPIstab}, the variables $(t/\eps,x/\eps)$ rescale the linearization around a homogeneous equilibrium for the Vlasov-Poisson equation and the equation with massless electrons, respectively, whereas here the rescaled profile $n(y)\mu(w)$ is localized, and the additional similarity drift appears as the order-$\eps$ perturbation in \eqref{eq:scaled}.
\end{remark}

\section{Construction of a localized unstable mode}
\label{sec:localized}

We construct an unstable mode by choosing the velocity profile $\mu$ and the spatial profile $n$ in the localized operator $\cL_{n,\mu}^{\rm loc}$ defined in \eqref{eq:Lloc}.

\subsection{The Birman-Schwinger reduction}
\label{sec:reduction}

Classical Penrose instabilities are obtained by linearizing around a spatially homogeneous profile and considering a single Fourier mode in the space variable. On $\R$, however, neither the background nor the corresponding mode is integrable, and a spatial cut-off cannot be used since it does not preserve the eigenvalue equation. We therefore introduce a spatial profile $n(y)$ and reduce the eigenvalue problem to a compact operator acting on the electric field.

More precisely, given $\Lambda>0$, we look for an eigenfunction $H$ of $\cL_{n,\mu}^{\rm loc}$ with eigenvalue $\Lambda$, and we denote its electric field by
\[
 E(y)=E[\rho_H](y),
\]
so that the eigenvalue equation can be written as
\be\label{eq:eigenvalue}
 (\Lambda+w\,\pt_y)H=-n(y)E(y)\mu'(w),
 \qquad
 E'=\rho_H.
\ee
We use the Fourier convention
\[
 \wh g(k):=\int_\R e^{-iky}g(y)\,\di y,
 \qquad
 g(y)=\frac1{2\pi}\int_\R e^{iky}\wh g(k)\,\di k,
\]
and, taking the Fourier transform in $y$, we obtain
\be\label{eq:Hhat}
 \wh H(k,w)
 =-\frac{\mu'(w)}{\Lambda+ikw}\,\wh{nE}(k).
\ee
For $k\ne0$, we then integrate by parts and find
\be\label{eq:parts}
 ik\,\wh E(k)=\wh{\rho_H}(k)
 =-\wh{nE}(k)\int_\R\frac{\mu'(w)}{\Lambda+ikw}\,\di w
 =-ik\,\wh{nE}(k)\int_\R\frac{\mu(w)}{(\Lambda+ikw)^2}\,\di w,
\ee
so division by $ik$ gives
\be\label{eq:multiplier}
 \wh E(k)=m_\Lambda(k)\,\wh{nE}(k),
 \qquad
 m_\Lambda(k)
 :=-\int_\R\frac{\mu(w)}{(\Lambda+ikw)^2}\,\di w.
\ee

Using
\[
 \frac{1}{(\Lambda+ikw)^2}
 =\int_0^\infty s e^{-\Lambda s}e^{-ikws}\,\di s,
\]
we may equivalently write
\be\label{eq:Laplace}
 m_\Lambda(k)=-\int_0^\infty s e^{-\Lambda s}\wh\mu(ks)\,\di s,
\ee
which is the Laplace transform of the response kernel appearing in the Volterra formulation of linear Landau damping; see for example \cite{BedrossianMasmoudiMouhotLD,GagnebinIacobelliLD,GrenierNguyenRodnianskiLD,MouhotVillani}.

Denoting by $\cT_\Lambda$ the Fourier multiplier with symbol $m_\Lambda$, we can then rewrite the eigenvalue problem in \eqref{eq:multiplier} as an equation for the electric field
\be\label{eq:fixed}
 E=\cT_\Lambda(nE).
\ee
A nonzero solution $E$ is therefore an eigenfunction of the field operator $E\mapsto\cT_\Lambda(nE)$ with eigenvalue $1$.

If $n\equiv n_0$, the Fourier modes decouple, and a nonzero mode of wave number $k$ exists precisely when
\[
 1=n_0m_\Lambda(k),
\]
which is the classical Penrose dispersion relation. For a nonconstant $n$, however,
\[
 \wh{nE}(k)=\frac1{2\pi}\int_\R\wh n(k-\ell)\wh E(\ell)\,\di\ell,
\]
so \eqref{eq:multiplier} relates different frequencies and no longer reduces to a separate algebraic condition at each $k$. The next proposition proves that the field operator is compact when $n$ is bounded and compactly supported.

A related reduction to a compact spatial operator is used in \cite[Section~8]{HRS22} to construct oscillatory modes around inhomogeneous equilibria of the gravitational Vlasov-Poisson equation.

\begin{prop}
\label{prop:reduction}
Let $\Lambda>0$. Assume $\mu\in W^{1,1}(\R)$, $\mu'\in L^2(\R)$, and $\mu\geq0$, and let $n\in L^\infty(\R)$ have compact support. Then $m_\Lambda\in L^2(\R)$ and the operator
\[
 E\longmapsto\cT_\Lambda(nE)
\]
is Hilbert-Schmidt, hence compact, on $L^2(\R)$. Moreover, nonzero solutions $E\in L^2(\R)$ of \eqref{eq:fixed} are in one-to-one correspondence with nonzero pairs $(H,E)$ such that
\[
 H\in {L^2_{y,w}}\cap L^1_wL^2_y,
 \qquad E\in L^2(\R),\qquad
 E'=\rho_H,\qquad
 (\Lambda+w\,\pt_y)H=-nE\mu'
\]
in the sense of distributions.
\end{prop}

\begin{proof}
For $|k|\leq1$, formula \eqref{eq:multiplier} gives $|m_\Lambda(k)|\leq\|\mu\|_{L^1}/\Lambda^2$, while for $|k|\geq1$ we use \eqref{eq:parts} to obtain
\[
 |m_\Lambda(k)|
 =\frac1{|k|}\lf|
 \int_\R\frac{\mu'(w)}{\Lambda+ikw}\,\di w
 \ri|
 \leq\frac{\|\mu'\|_{L^1}}{\Lambda|k|},
\]
so $m_\Lambda\in L^2(\R)$, and hence, by Plancherel, the convolution kernel $K_\Lambda=\mathcal F^{-1}m_\Lambda$ also belongs to $L^2(\R)$. Since the map $E\mapsto\cT_\Lambda(nE)$ has kernel $K_\Lambda(y-z)n(z)$, we compute
\[
 \iint_{\R^2}|K_\Lambda(y-z)n(z)|^2\,\di y\,\di z
 =\|K_\Lambda\|_{L^2}^2\|n\|_{L^2}^2<\infty,
\]
which proves that this operator is Hilbert-Schmidt.
{In particular, it is compact by \cite[Chapter~II, \S4]{Conway}.}

If $E$ solves \eqref{eq:fixed}, define $H$ by
\[
 \wh H(k,w)=-\frac{\mu'(w)}{\Lambda+ikw}\,\wh{nE}(k).
\]
Since $n$ is bounded, $nE$ belongs to $L^2$, so we can use Plancherel's identity and $|\Lambda+ikw|\geq\Lambda$ to obtain
\[
 \|H(\cdot,w)\|_{L^2_y}
 \leq \Lambda^{-1}|\mu'(w)|\,\|nE\|_{L^2_y},
\]
so, taking the $L^2$ norm in $w$, we find
\[
 \|H\|_{L^2_{y,w}}
 \leq \Lambda^{-1}\|\mu'\|_{L^2_w}\|nE\|_{L^2_y}.
\]
Considering instead the $L^1$ norm in $w$, we get
\[
 \|H\|_{L^1_wL^2_y}
 \leq \Lambda^{-1}\|\mu'\|_{L^1_w}\|nE\|_{L^2_y},
\]
and therefore the spatial density satisfies
\[
 \|\rho_H\|_{L^2_y}
 \leq\int_\R\|H(\cdot,w)\|_{L^2_y}\,\di w
 =\|H\|_{L^1_wL^2_y}.
\]
Formula \eqref{eq:Hhat} gives the kinetic equation in \eqref{eq:eigenvalue}, and using \eqref{eq:parts} we also obtain $\wh{\rho_H}(k)=ik\,m_\Lambda(k)\wh{nE}(k)=ik\,\wh E(k)$, which is the Fourier form of $E'=\rho_H$.
Conversely, for every pair in the statement, we recover \eqref{eq:multiplier} from \eqref{eq:parts}, and hence \eqref{eq:fixed}. If $E=0$, then \eqref{eq:eigenvalue} and $\Lambda>0$ force $H=0$.
\end{proof}

\subsection{{Choice of the velocity profile}}
\label{sec:cauchy}

The idea of using the two-stream instability in our non-uniqueness construction is inspired by our work on the quasineutral limit, and in particular by the instability argument in \cite[Appendix~B]{GPIstab}.
In that work, the instability is used to prove the failure of the quasineutral limit, whereas here we construct a localized unstable mode and use its growth in self-similar time to obtain non-uniqueness.

For every probability density $\mu$, the multiplier in \eqref{eq:multiplier} satisfies
\[
 m_\Lambda(0)=-\Lambda^{-2}.
\]
We look for a profile $\mu$ of the form
\[
 \mu(w)=\frac12\bigl(g(w-u)+g(w+u)\bigr),
\]
where $g$ is a smooth, strictly positive and even probability density, chosen so that $m_\Lambda$ becomes positive at high frequencies and $\mu$ has a finite first moment and satisfies the derivative bounds needed below.

Inspired by computations of dispersion relations for two-stream profiles in the plasma physics literature (see, for example, \cite{RightleyUzdensky}), we take $g$ to be the normalized square of a Cauchy density, so that $g$ has a finite first moment and an explicit Fourier transform.

\begin{defin}\label{def:cauchy}
Let $u,\gamma>0$. We define the centered Cauchy distribution
\[
 C_\gamma(w):=\frac{\gamma}{\pi(w^2+\gamma^2)},
\]
its normalized square
\[
 L_\gamma(w):=2\pi\gamma C_\gamma(w)^2
 =\frac{2\gamma^3}{\pi(w^2+\gamma^2)^2},
\]
and the symmetric two-stream profile $\mu_{u,\gamma}$
\be\label{eq:mu}
 \mu_{u,\gamma}(w)
 :=\frac12\bigl(L_\gamma(w-u)+L_\gamma(w+u)\bigr).
\ee
\end{defin}

The centered profile $L_1$ also appears as a one-dimensional marginal of the generalized Poisson equilibria in \cite[Section~2.1]{IPPW23}, where the associated Fourier-Laplace kernels are computed explicitly.

\begin{lemma}\label{lem:profiles}
Let $u,\gamma>0$, and let $C_\gamma$, $L_\gamma$ and $\mu_{u,\gamma}$ be as in Definition~\ref{def:cauchy}. The following properties hold.
\begin{itemize}
\item[(i)] {The functions $C_\gamma,L_\gamma,\mu_{u,\gamma}\in C^\infty(\R)$ are strictly positive.}
\item[(ii)] The functions $C_\gamma$, $L_\gamma$ and $\mu_{u,\gamma}$ are even.
\item[(iii)] The three functions have unit mass, namely
\[
 \int_\R C_\gamma(w)\,\di w
 =\int_\R L_\gamma(w)\,\di w
 =\int_\R\mu_{u,\gamma}(w)\,\di w=1.
\]
\item[(iv)] The first moments satisfy
\[
 \int_\R |w|C_\gamma(w)\,\di w=+\infty,\qquad
 \int_\R |w|L_\gamma(w)\,\di w=\frac{2\gamma}{\pi},
 \qquad
 \int_\R |w|\mu_{u,\gamma}(w)\,\di w
 \leq u+\frac{2\gamma}{\pi}.
\]
\item[(v)] As $|w|\to\infty$,
\[
 \mu_{u,\gamma}(w)=O(|w|^{-4}),\qquad
 \mu_{u,\gamma}'(w)=O(|w|^{-5}),\qquad
 \mu_{u,\gamma}''(w)=O(|w|^{-6}).
\]
In particular,
\be\label{eq:moments}
 \int_\R(1+|w|)\mu_{u,\gamma}(w)\,\di w<\infty,
 \qquad
 \mu_{u,\gamma}',\;w\mu_{u,\gamma}'\in L^1(\R)\cap L^2(\R),
 \qquad \mu_{u,\gamma}''\in L^1(\R).
\ee
\item[(vi)] With the convention
\[
 \widehat\varphi(\eta):=\int_\R e^{-i\eta w}\varphi(w)\,\di w,
\]
we have
\be\label{eq:Fourier}
 \widehat C_\gamma(\eta)=e^{-\gamma|\eta|},\qquad
 \widehat L_\gamma(\eta)=(1+\gamma|\eta|)e^{-\gamma|\eta|},\qquad
 \widehat\mu_{u,\gamma}(\eta)
 =(1+\gamma|\eta|)e^{-\gamma|\eta|}\cos(u\eta).
\ee
\end{itemize}
\end{lemma}

\begin{proof}
The formulas in Definition~\ref{def:cauchy} prove (i). Since $C_\gamma$ and $L_\gamma$ are even,
\[
 \mu_{u,\gamma}(-w)
 =\frac12\bigl(L_\gamma(w+u)+L_\gamma(w-u)\bigr)
 =\mu_{u,\gamma}(w),
\]
which proves (ii).

For (iii), changing variables $w=\gamma z$, we obtain
\[
 \int_\R C_\gamma(w)\,\di w
 =\frac1\pi\int_\R\frac{1}{1+z^2}\,\di z=1,
 \qquad
 \int_\R C_\gamma(w)^2\,\di w=\frac{1}{2\pi\gamma}.
\]
It follows from the definition of $L_\gamma$ and invariance under translations that
\[
 \int_\R L_\gamma(w)\,\di w
 =\int_\R\mu_{u,\gamma}(w)\,\di w=1.
\]

To prove (iv), we first observe that
\[
 \int_{-R}^R|w|C_\gamma(w)\,\di w
 =\frac{\gamma}{\pi}\log\left(\frac{R^2+\gamma^2}{\gamma^2}\right),
\]
which diverges as $R\to\infty$. Considering $L_\gamma$ instead, we compute
\[
 \int_\R |w|L_\gamma(w)\,\di w
 =\frac{4\gamma^3}{\pi}\int_0^\infty\frac{w}{(w^2+\gamma^2)^2}\,\di w
 =\frac{2\gamma}{\pi},
\]
and, after translating each copy of $L_\gamma$ and using the triangle inequality, we conclude that
\[
 \int_\R |w|\mu_{u,\gamma}(w)\,\di w
 \leq u+\int_\R |w|L_\gamma(w)\,\di w,
\]
which is the last estimate in (iv).

For (v), differentiating $L_\gamma$ gives
\be\label{eq:Cauchy}
 L_\gamma'(w)
 =-\frac{8\gamma^3w}{\pi(w^2+\gamma^2)^3},
 \qquad
 L_\gamma''(w)
 =\frac{8\gamma^3(5w^2-\gamma^2)}{\pi(w^2+\gamma^2)^4}.
\ee
The three decay estimates follow from \eqref{eq:Cauchy} and \eqref{eq:mu}.
To deduce \eqref{eq:moments}, we first combine (iii) and (iv) to get the finite first moment.
Since $\mu_{u,\gamma}'$ and $w\mu_{u,\gamma}'$ are bounded on compact sets and decay respectively as $|w|^{-5}$ and $|w|^{-4}$, they belong to $L^1(\R)\cap L^2(\R)$, and the same argument applied to $\mu_{u,\gamma}''$ proves that $\mu_{u,\gamma}''\in L^1(\R)$.

Finally, computing the residues for the Cauchy distribution, we obtain
\[
 \widehat C_\gamma(\eta)
 =\frac1\pi\int_\R\frac{e^{-i\gamma\eta z}}{1+z^2}\,\di z
 =e^{-\gamma|\eta|}.
\]
Differentiating with respect to $\gamma$, we find
\[
 C_\gamma(w)-\gamma\pt_\gamma C_\gamma(w)=L_\gamma(w).
\]
Taking the Fourier transform and then applying the translation formula, we can write
\[
 \bal
 \widehat L_\gamma(\eta)
 &=\widehat C_\gamma(\eta)-\gamma\pt_\gamma\widehat C_\gamma(\eta)
 =(1+\gamma|\eta|)e^{-\gamma|\eta|},\\
 \widehat\mu_{u,\gamma}(\eta)
 &=\frac12\bigl(e^{-iu\eta}+e^{iu\eta}\bigr)\widehat L_\gamma(\eta)
 =(1+\gamma|\eta|)e^{-\gamma|\eta|}\cos(u\eta),
 \eal
\]
which proves (vi).
\end{proof}

We now compute $m_\Lambda$ and show that it is positive at high frequencies when $u>\gamma$, which will allow us to choose the spatial profile in the next subsection.

\begin{lemma}
\label{lem:multiplier}
Let $u,\gamma,\Lambda>0$. For $k\in\R$, write
\[
 A:=\Lambda+\gamma|k|,\qquad B:=A^2+u^2k^2.
\]
Then
\be\label{eq:symbol}
 m_\Lambda(k)
 =\frac{u^2k^2-A^2}{B^2}
 +2\gamma|k|A\,\frac{3u^2k^2-A^2}{B^3}.
\ee
Moreover,
\be\label{eq:asymptotic}
 m_\Lambda(k)
 =\frac{u^4+6\gamma^2u^2-3\gamma^4}
 {(u^2+\gamma^2)^3}\,\frac1{|k|^2}
 +O(|k|^{-3}).
\ee
In particular, if $u>\gamma$, then
\be\label{eq:positive}
 m_\Lambda(k)>0
 \qquad\text{for all sufficiently large }|k|.
\ee
\end{lemma}

\begin{proof}
We use \eqref{eq:multiplier} and \eqref{eq:Laplace} to write, for the profile $\mu_{u,\gamma}$ defined in \eqref{eq:mu},
\[
 \bal
 m_\Lambda(k)
 &=-\int_\R\frac{\mu_{u,\gamma}(w)}{(\Lambda+ikw)^2}\,\di w\\
 &=-\int_0^\infty s e^{-\Lambda s}\wh\mu_{u,\gamma}(ks)\,\di s,
 \eal
\]
so, using \eqref{eq:Fourier} and the notation in the statement, we obtain
\be\label{eq:mu Laplace}
 m_\Lambda(k)
 =-\int_0^\infty s e^{-As}(1+\gamma|k|s)\cos(u|k|s)\,\di s.
\ee
Differentiating the identity
\[
 \int_0^\infty e^{-As}\cos(u|k|s)\,\di s=\frac{A}{B}
\]
once and twice with respect to $A$, keeping $u|k|$ fixed, we find
\be\label{eq:cosine}
 \int_0^\infty s e^{-As}\cos(u|k|s)\,\di s
 =\frac{A^2-u^2k^2}{B^2},
 \qquad
 \int_0^\infty s^2 e^{-As}\cos(u|k|s)\,\di s
 =\frac{2A(A^2-3u^2k^2)}{B^3}.
\ee
Substituting \eqref{eq:cosine} into \eqref{eq:mu Laplace} proves \eqref{eq:symbol}. Since $A=\gamma|k|+\Lambda$ and $B=(u^2+\gamma^2)|k|^2+2\gamma\Lambda|k|+\Lambda^2$, we have
\[
 \frac{u^2k^2-A^2}{B^2}
 =\frac{u^2-\gamma^2}{(u^2+\gamma^2)^2}\frac1{|k|^2}
 +O(|k|^{-3})
\]
and
\[
 2\gamma|k|A\frac{3u^2k^2-A^2}{B^3}
 =\frac{2\gamma^2(3u^2-\gamma^2)}{(u^2+\gamma^2)^3}\frac1{|k|^2}
 +O(|k|^{-3}),
\]
so, adding these two expansions, we obtain \eqref{eq:asymptotic}, whose leading coefficient is positive when
\[
 \frac{u^2}{\gamma^2}>-3+2\sqrt3,
\]
which holds, in particular, if $u>\gamma$, and therefore \eqref{eq:positive} follows from \eqref{eq:asymptotic}.
\end{proof}

\begin{remark}
If $u=0$, instead,
\[
 m_\Lambda(k)
 =-\frac1{(\Lambda+\gamma|k|)^2}
 -\frac{2\gamma|k|}{(\Lambda+\gamma|k|)^3}<0,
\]
so the multiplier is negative for every $k$, as predicted by the Penrose stability criterion \cite{Penrose60}.
\end{remark}

{
\begin{remark}
The leading coefficient in \eqref{eq:asymptotic} is the Penrose integral of $\mu_{u,\gamma}$ at $w=0$. Indeed, integrating by parts in \eqref{eq:multiplier}, we obtain, for $k\ne0$,
\[
 m_\Lambda(k)=\frac1{k^2}\int_\R\frac{\mu_{u,\gamma}'(w)}{w-i\Lambda/k}\,\di w.
\]
Since $\mu_{u,\gamma}$ is smooth and even, and its derivative has the decay stated in Lemma~\ref{lem:profiles}, the function $|\mu_{u,\gamma}'(w)|/|w|$ is integrable. Dominated convergence and another integration by parts give
\[
 \bal
 \lim_{|k|\to\infty}|k|^2m_\Lambda(k)
 &=\int_\R\frac{\mu_{u,\gamma}'(w)}{w}\,\di w
 =\int_\R\frac{\mu_{u,\gamma}(w)-\mu_{u,\gamma}(0)}{w^2}\,\di w\\
 &=\frac{u^4+6\gamma^2u^2-3\gamma^4}{(u^2+\gamma^2)^3}.
 \eal
\]
When $u>\gamma$, the point $w=0$ is a strict local minimum of $\mu_{u,\gamma}$ and this integral is positive, so the profile satisfies the Penrose instability condition on the whole line \cite{Penrose60,HKH15}.
\end{remark}
}

\subsection{Choice of the spatial profile}
\label{sec:space}

Having fixed the velocity profile $\mu_{u,\gamma}$, we now choose $n$ so that $\cL_{n,\mu_{u,\gamma}}^{\rm loc}$ has a prescribed positive eigenvalue.
To use the positivity of $m_\Lambda$ in \eqref{eq:positive}, we first show that the field operator $E\mapsto\cT_\Lambda(nE)$ and the compact self-adjoint operator $\sqrt n\,\cT_\Lambda\sqrt n$ have the same nonzero eigenvalues. We then choose $n$ so that the quadratic form of the latter is positive on a suitable test function.

\begin{lemma}
\label{lem:symmetrization}
Let $\Lambda>0$ and $u>\gamma>0$, and let $n\in L^\infty(\R)$ be nonnegative and compactly supported.
For the velocity profile $\mu_{u,\gamma}$, the following properties hold.
\begin{itemize}
\item[(i)] The operator $\sqrt n\cT_\Lambda\sqrt n$ is compact and self-adjoint on $L^2(\R)$.
\item[(ii)] Its nonzero eigenvalues coincide with those of $E\mapsto\cT_\Lambda(nE)$.
Moreover, a simple nonzero eigenvalue of the symmetrized operator is algebraically simple for the field operator.
For each $\nu\ne0$, the correspondence between eigenvectors is given by
\be\label{eq:ABBA}
 g=\sqrt n\,E,
 \qquad E=\nu^{-1}\cT_\Lambda(\sqrt n\,g).
\ee
\end{itemize}
\end{lemma}

\begin{proof}
To prove (i), we first observe that $m_\Lambda$ is bounded and, by the evenness of $\mu_{u,\gamma}$, real-valued.
It follows that $\cT_\Lambda$ is a bounded self-adjoint operator on $L^2(\R)$. Since $n$ and $\sqrt n$ are bounded and compactly supported, Proposition~\ref{prop:reduction} also implies that $E\mapsto\cT_\Lambda(nE)$ and $g\mapsto\cT_\Lambda(\sqrt n\,g)$ are Hilbert-Schmidt, and since multiplication by $\sqrt n$ is bounded, we deduce that $\sqrt n\cT_\Lambda\sqrt n$ is compact.
Moreover,
\[
 \bigl(\sqrt n\cT_\Lambda\sqrt n\bigr)^*
 =\sqrt n\cT_\Lambda^*\sqrt n
 =\sqrt n\cT_\Lambda\sqrt n,
\]
which proves self-adjointness.

To prove that the field operator and the symmetrized operator have the same nonzero eigenvalues, we use the $AB$-$BA$ correspondence (see, for example, {\cite[Chapter~VII, \S3]{Conway}} and \cite[Chapter~10]{Rudin}) and set
\[
 A\varphi:=\cT_\Lambda(\sqrt n\,\varphi),
 \qquad B\varphi:=\sqrt n\,\varphi,
\]
so that $ABE=\cT_\Lambda(nE)$ and $BAg=\sqrt n\,\cT_\Lambda(\sqrt n\,g)$.
Let $\nu\ne0$. If $ABE=\nu E$, we set $g:=BE$ and obtain
\[
 BAg=B(ABE)=\nu g,\qquad Ag=\nu E.
\]
Conversely, if $BAg=\nu g$, setting $E:=\nu^{-1}Ag$ gives
\[
 ABE=\nu^{-1}A(BAg)=\nu E,\qquad BE=\nu^{-1}BAg=g,
\]
so these maps are inverse to each other on the eigenspaces, which proves the claimed correspondence.

Suppose now that $\nu$ is a simple nonzero eigenvalue of $BA$. Since $BA$ is self-adjoint, $\nu$ is real.
If $(AB-\nu I)^2E=0$, then
\[
 \begin{aligned}
 &(BA-\nu I)^2BE=B(AB-\nu I)^2E=0,\\
 &\|(BA-\nu I)BE\|_{L^2}^2
 =\langle BE,(BA-\nu I)^2BE\rangle=0.
 \end{aligned}
\]
Using again $(AB-\nu I)^2E=0$, we find
\[
 \nu(AB-\nu I)E=AB(AB-\nu I)E=A(BA-\nu I)BE=0,
\]
and since $\nu\ne0$, we deduce that $\ker((AB-\nu I)^2)=\ker(AB-\nu I)$, so the generalized eigenspace coincides with the eigenspace. The latter is one-dimensional by the correspondence in \eqref{eq:ABBA}, and therefore $\nu$ is algebraically simple.
\end{proof}

{
By Lemma~\ref{lem:symmetrization}, to obtain a positive eigenvalue it is enough to choose $n$ and $g$ such that
\[
 \langle g,\sqrt n\,\cT_\Lambda(\sqrt n\,g)\rangle>0.
\]
We use the positivity of $m_\Lambda$ at high frequencies to construct such a test function.

}

\begin{lemma}
\label{lem:positive}
Fix $\Lambda>0$ and $u>\gamma>0$. There exists a nonnegative $\chi\in C_c^\infty(\R)$ such that, for $n:=\chi^2$, the compact self-adjoint operator $\sqrt n\cT_\Lambda\sqrt n$ has a positive eigenvalue.
\end{lemma}

\begin{proof}
We choose a nonzero real function $\phi\in C_c^\infty(\R)$ and set $\psi_{k_0}(y):=\phi(y)e^{ik_0y}$. By Plancherel,
\be\label{eq:rayleigh}
 \langle \psi_{k_0},\cT_\Lambda \psi_{k_0}\rangle
 =\frac1{2\pi}\int_\R
 m_\Lambda(k)\,|\wh\phi(k-k_0)|^2\,\di k.
\ee
By \eqref{eq:asymptotic}, there exist $K,C_1>0$ such that
\[
 m_\Lambda(k)\geq C_1|k|^{-2}\qquad\text{for }|k|\geq K.
\]
Choose $R>0$ such that $\int_{-R}^R|\wh\phi(k)|^2\,\di k>0$.
Since $\wh\phi$ is a Schwartz function, for every integer $N\geq1$ and all sufficiently large $k_0$,
\[
 \int_{|k|\leq K}|m_\Lambda(k)|\,|\wh\phi(k-k_0)|^2\,\di k
 \leq C_N k_0^{-N},
\]
while
\[
 \int_{k_0-R}^{k_0+R}m_\Lambda(k)\,|\wh\phi(k-k_0)|^2\,\di k
 \geq\frac{C_1}{(k_0+R)^2}\int_{-R}^R|\wh\phi(k)|^2\,\di k
 \geq C_2 k_0^{-2}
\]
for some $C_2>0$ independent of $k_0$. Since the remaining frequencies outside $[-K,K]$ make a nonnegative contribution, taking $N>2$ in \eqref{eq:rayleigh} gives
\[
 2\pi\langle\psi_{k_0},\cT_\Lambda\psi_{k_0}\rangle
 \geq C_2k_0^{-2}-C_Nk_0^{-N}>0
\]
for all sufficiently large $k_0$.

Let $\chi\in C_c^\infty(\R)$, $\chi\geq0$, with $\chi>0$ on $\spt\phi$, and put $n:=\chi^2$. Taking $g:=\psi_{k_0}/\chi$ on $\spt\phi$ and $g:=0$ outside, we obtain
\[
 \langle g,\sqrt n\cT_\Lambda\sqrt ng\rangle
 =\langle \psi_{k_0},\cT_\Lambda \psi_{k_0}\rangle>0,
\]
and since the operator is compact and self-adjoint by Lemma~\ref{lem:symmetrization}(i), it has a positive eigenvalue {by the spectral theorem \cite[Chapter~II, \S5]{Conway}}.
\end{proof}

{We next show that, after a small perturbation of $n$, the symmetrized operator has a simple positive eigenvalue.}

\begin{lemma}
\label{lem:simplicity}
Fix $\Lambda>0$ and $u>\gamma>0$. There exists $\widetilde n\in C_c^\infty(\R)$, $\widetilde n\geq0$, such that the compact self-adjoint operator $\sqrt{\widetilde n}\cT_\Lambda\sqrt{\widetilde n}$ has a simple positive eigenvalue $\nu_*$. 
\end{lemma}

\begin{proof}
Let $n=\chi^2$ be given by Lemma~\ref{lem:positive}. If the symmetrized operator has a simple positive eigenvalue, there is nothing to prove. Otherwise, we choose a positive eigenvalue $\nu>0$ of multiplicity $\ell\geq2$ and replace $\chi$ by $\chi+t\chi_1$, where $\chi_1$ is real, smooth, and compactly supported in a set where $\chi>0$, taking $|t|$ small enough that $\chi+t\chi_1\geq0$. The resulting family $(\chi+t\chi_1)\cT_\Lambda(\chi+t\chi_1)$ depends analytically on $t$. Since the multiplier in \eqref{eq:symbol} is real and even, we can choose a real orthonormal basis $e_1,\ldots,e_\ell$ of the $\nu$-eigenspace, and differentiating at $t=0$ and using $\chi\cT_\Lambda(\chi e_j)=\nu e_j$ gives
\be\label{eq:splitting}
 \bal
 &\left.\frac{\di}{\di t}\right|_{t=0}
 \bigl\langle e_i,(\chi+t\chi_1)\cT_\Lambda((\chi+t\chi_1)e_j)\bigr\rangle\\
 &\qquad=\langle \chi_1e_i,\cT_\Lambda(\chi e_j)\rangle
 +\langle\cT_\Lambda(\chi e_i),\chi_1e_j\rangle
 =2\nu\int_\R\frac{\chi_1}{\chi}\,e_i e_j\,\di y,
 \eal
\ee
where $\chi_1/\chi$ is extended by zero outside the support of $\chi_1$.

The matrix in \eqref{eq:splitting} cannot be scalar for every $\chi_1$. Indeed, by varying $\chi_1$ we would obtain, for $i\ne j$, almost everywhere on $\{\chi>0\}$,
\[
 e_i e_j=0,\qquad e_i^2=e_j^2
 \quad\Longrightarrow\quad
 e_i^4=e_i^2e_j^2=0
 \quad\Longrightarrow\quad e_i=e_j=0,
\]
while the eigenvalue equation also implies $e_i=e_j=0$ on $\{\chi=0\}$, contradicting their normalization.
We can therefore choose $\chi_1$ for which the matrix is not scalar, and by \cite[Chapter~II, Theorem~5.4, and Chapter~VII, \S3.1]{Kato} the eigenvalues near $\nu$ can be written as
\[
 \nu_j(t)=\nu+t\sigma_j+O(t^2),\qquad j=1,\ldots,\ell,
\]
where $\sigma_1,\ldots,\sigma_\ell$ are the eigenvalues of the matrix in \eqref{eq:splitting}.
At least two of them are distinct, so for small nonzero $t$ the eigenvalues remain positive and separate into groups of total multiplicity smaller than $\ell$. If an eigenvalue in one group is still multiple, we repeat the argument, decreasing its multiplicity at each step until we obtain a profile $\widetilde n$ with a simple positive eigenvalue $\nu_*$.
\end{proof}

To apply Proposition~\ref{prop:reduction}, we still have to normalize the simple positive eigenvalue to $1$. Since the field operator depends linearly on $n$, we can do this by adjusting the amplitude of the spatial profile while keeping the kinetic growth rate fixed.

\begin{theorem}
\label{thm:mode}
Fix $\Lambda_0>0$ and $u>\gamma>0$. There exist a nonnegative $n_*\in C_c^\infty(\R)$ and a nonzero real-valued function
\[
 H\in {L^2_{y,w}\cap L^1_wL^2_y}
\]
such that
\[
 \cL_{n_*,\mu_{u,\gamma}}^{\rm loc}H=\Lambda_0H
\]
in the sense of distributions, and its electric field $E[\rho_H]$ belongs to $L^2(\R)$. Moreover, the map $E\mapsto\cT_{\Lambda_0}(n_*E)$ has the algebraically simple eigenvalue $1$.
\end{theorem}

\begin{proof}
Applying Lemma~\ref{lem:simplicity} with $\Lambda=\Lambda_0$, we find a profile $\widetilde n$ such that $\sqrt{\widetilde n}\cT_{\Lambda_0}\sqrt{\widetilde n}$ has a simple positive eigenvalue $\nu_*$. By Lemma~\ref{lem:symmetrization}(ii), this is also an algebraically simple eigenvalue of $E\mapsto\cT_{\Lambda_0}(\widetilde nE)$, so we normalize it to $1$ by setting
\[
 n_*:=\nu_*^{-1}\widetilde n,
\]
and since $E\mapsto\cT_{\Lambda_0}(nE)$ depends linearly on $n$, the map $E\mapsto\cT_{\Lambda_0}(n_*E)$ has the algebraically simple eigenvalue $1$. By Proposition~\ref{prop:reduction}, we reconstruct a nonzero eigenfunction of $\cL_{n_*,\mu_{u,\gamma}}^{\rm loc}$ in ${L^2_{y,w}\cap L^1_wL^2_y}$ with eigenvalue $\Lambda_0$, whose electric field belongs to $L^2(\R)$.
Finally, since the kinetic operator and $\Lambda_0$ are real, at least one of the real and imaginary parts of the eigenfunction is nonzero and is a real-valued eigenfunction.
\end{proof}

\section{Restoring the similarity drift}
\label{sec:drift}

We now include the order-$\eps$ similarity drift in \eqref{eq:scaled}. We first prove convergence in operator norm of the field operators, so the simple eigenvalue persists near $1$, and then adjust the amplitude of the spatial profile to restore the field eigenvalue to $1$ while keeping the rescaled kinetic eigenvalue fixed for small $\eps$.
{For the persistence of instability under the self-similar drift in the Euler equation, see \cite{ABCDLGJK}.}
Writing $\Lambda=\eps\lambda$, we define
\[
 \mathcal G_\eps
 :=-w\,\pt_y+\eps\mathcal D_{b,c}
 =-w\,\pt_y+\eps b\,y\,\pt_y+\eps c\,w\,\pt_w-\eps a,
\]
so that the eigenvalue equation becomes
\[
 (\Lambda-\mathcal G_\eps)H=-n(y)E(y)\mu'(w).
\]
We solve this equation using the resolvent of $\mathcal G_\eps$, which we compute along the characteristics in the following lemma.

\begin{lemma}
\label{lem:resolvent}
Let $\eps>0$, $b=c+1$ and $a=-2-c$. The following properties hold.
\begin{itemize}
\item[(i)] The characteristic flow is
\be\label{eq:flow}
 w_s=e^{\eps cs}w,
 \qquad
 y_s=e^{\eps bs}\lf(
 y-w\,\frac{1-e^{-\eps s}}{\eps}
 \ri).
\ee
The operators
\be\label{eq:transport}
 e^{s\mathcal G_\eps}\Psi(y,w):=e^{-\eps as}\Psi(y_s,w_s),\qquad s\geq0,
\ee
form a strongly continuous semigroup on ${L^2_{y,w}}$. Its generator agrees with $\mathcal G_\eps$ on $C_c^\infty(\R^2)$ and has domain
\[
 D(\mathcal G_\eps)
 =\bigl\{H\in {L^2_{y,w}}:\mathcal G_\eps H\in {L^2_{y,w}}
 \text{ in the sense of distributions}\bigr\}.
\]
Moreover,
\be\label{eq:semigroup}
 \|e^{s\mathcal G_\eps}\|_{L^2\to L^2}
 =e^{-\eps(a+(b+c)/2)s}=e^{3\eps s/2}.
\ee
\item[(ii)] For $\Lambda>3\eps/2$ and $\Psi\in {L^2_{y,w}}$, the resolvent is
\be\label{eq:resolvent}
 (\Lambda-\mathcal G_\eps)^{-1}\Psi(y,w)
 =\int_0^\infty e^{-(\Lambda+\eps a)s}\Psi(y_s,w_s)\,\di s,
\ee
where the integral converges in ${L^2_{y,w}}$, and
\be\label{eq:resolvent norm}
 \|(\Lambda-\mathcal G_\eps)^{-1}\Psi\|_{L^2}
 \leq\frac{\|\Psi\|_{L^2}}{\Lambda-3\eps/2}.
\ee
\end{itemize}
\end{lemma}

\begin{proof}
We solve the characteristic equations
\[
 \dot w_s=\eps c w_s,\qquad
 \dot y_s=\eps b y_s-w_s,\qquad (y_0,w_0)=(y,w),
\]
and, using $b-c=1$, we obtain \eqref{eq:flow} by integration. The Jacobian of the map $(y,w)\mapsto(y_s,w_s)$ is $e^{\eps(b+c)s}$, and therefore
\[
 \|e^{-\eps as}\Psi(y_s,w_s)\|_{L^2_{y,w}}^2
 =e^{-\eps(2a+b+c)s}\|\Psi\|_{L^2}^2
 =e^{3\eps s}\|\Psi\|_{L^2}^2.
\]
The flow property implies the semigroup property, and for smooth compactly supported $\Psi$ the formula \eqref{eq:transport} is continuous in $L^2$ at $s=0$, with derivative $\mathcal G_\eps\Psi$. Using density and \eqref{eq:semigroup}, we then conclude that the semigroup is strongly continuous on $L^2$.
{If $H,\mathcal G_\eps H\in L^2$,} testing \eqref{eq:transport} against smooth compactly supported functions gives
\[
 e^{s\mathcal G_\eps}H-H
 =\int_0^s e^{r\mathcal G_\eps}(\mathcal G_\eps H)\,\di r,
\]
so, dividing by $s$ and using strong continuity, we deduce that {$H\in D(\mathcal G_\eps)$}.
{Conversely, if $H\in D(\mathcal G_\eps)$, the quotient $(e^{s\mathcal G_\eps}H-H)/s$ converges in $L^2$ as $s\to0^+$, and by testing against smooth compactly supported functions we see that its limit equals $\mathcal G_\eps H$ in the sense of distributions.}

For (ii), we first take $\Psi\in C_c^\infty(\R^2)$. Integrating the derivative of $e^{-\Lambda s}e^{s\mathcal G_\eps}\Psi$, we find
\[
 (\Lambda-\mathcal G_\eps)
 \int_0^R e^{-\Lambda s}e^{s\mathcal G_\eps}\Psi\,\di s
 =\Psi-e^{-\Lambda R}e^{R\mathcal G_\eps}\Psi.
\]
When $\Lambda>3\eps/2$, estimate \eqref{eq:semigroup} shows that the last term tends to zero and the integral converges in $L^2$, so the closedness of the generator implies that the limit solves $(\Lambda-\mathcal G_\eps)H=\Psi$. To check uniqueness, we observe that $(\Lambda-\mathcal G_\eps)H=0$ implies $e^{s\mathcal G_\eps}H=e^{\Lambda s}H$, which forces $H=0$ by \eqref{eq:semigroup}. Finally, we estimate
\[
 \left\|\int_0^\infty e^{-\Lambda s}e^{s\mathcal G_\eps}\Psi\,\di s\right\|_{L^2}
 \leq\|\Psi\|_{L^2}\int_0^\infty e^{-(\Lambda-3\eps/2)s}\,\di s
 =\frac{\|\Psi\|_{L^2}}{\Lambda-3\eps/2},
\]
and then use this bound and the closedness of the generator to extend \eqref{eq:resolvent} to every $\Psi\in {L^2_{y,w}}$ by approximation with smooth compactly supported functions.
\end{proof}

To retain the simple field eigenvalue when $\eps$ is small, we need convergence in operator norm, so we derive the kernels of the field operators from \eqref{eq:resolvent}. For $s>0$, define
\[
 d_\eps(s):=\frac{1-e^{-\eps s}}{\eps},
\]
and, for compactly supported $\varphi\in L^2(\R)$, set
\be\label{eq:Teps}
 (\cT_{\eps,\Lambda}\varphi)(y)
 :=\int_\R K_{\eps,\Lambda}(y,z)\,\varphi(z)\,\di z,
\ee
and
\be\label{eq:kernel}
 K_{\eps,\Lambda}(y,z)
 :=-\int_0^\infty
 e^{-(\Lambda+\eps(a+b+c))s}\,
 \mu\lf(
 \frac{e^{\eps cs}y-e^{-\eps s}z}{d_\eps(s)}
 \ri)\,\di s.
\ee
For $\eps=0$, define
\be\label{eq:K0}
 K_{0,\Lambda}(y,z)
 :=-\int_0^\infty e^{-\Lambda s}\,
 \mu\lf(\frac{y-z}{s}\ri)\,\di s.
\ee
\begin{lemma}
\label{lem:kernels}
Let $\mu=\mu_{u,\gamma}$, $n\in C_c^\infty(\R)$, and let $J\subset(0,\infty)$ be a compact interval.
For all sufficiently small $\eps>0$ and every $\Lambda\in J$, the following properties hold.
\begin{itemize}
\item[(i)] The operators
\[
 E\longmapsto\cT_{\eps,\Lambda}(nE)
\]
are Hilbert-Schmidt on $L^2(\R)$. At $\eps=0$, we have $\cT_{0,\Lambda}=\cT_\Lambda$.
\item[(ii)] We have
\be\label{eq:Klimit}
 \sup_{\Lambda\in J}
 \|\cT_{\eps,\Lambda}(n\,\cdot)-\cT_{0,\Lambda}(n\,\cdot)\|_{L^2\to L^2}
 \longrightarrow0
 \qquad\text{as }\eps\to0^+.
\ee
\item[(iii)] For every compactly supported $\varphi\in L^2(\R)$, the function
\[
 H:=-(\Lambda-\mathcal G_\eps)^{-1}\bigl(\varphi(y)\mu'(w)\bigr)
\]
belongs to ${L^2_{y,w}}\cap L^1_wL^2_y$, and
\[
 \cT_{\eps,\Lambda}\varphi\in L^2(\R),\qquad
 \pt_y(\cT_{\eps,\Lambda}\varphi)=\rho_H
\]
in the sense of distributions. In particular, for $\varphi=nE$ and $E\in L^2(\R)$, $E'=\rho_H$ if and only if $E=\cT_{\eps,\Lambda}(nE)$.
\end{itemize}
\end{lemma}

\begin{proof}
We first prove (i) and (ii). The operator $E\mapsto\cT_{\eps,\Lambda}(nE)$ has kernel $K_{\eps,\Lambda}(y,z)n(z)$. Let $I$ be a fixed compact interval containing $\spt n$, and set
\[
 \beta_\eps(s):=e^{-\eps bs},
\]
so that, to estimate the kernel in $L^2_y$, we can write the argument of $\mu$ in \eqref{eq:kernel} as $\frac{e^{\eps cs}}{d_\eps(s)}(y-\beta_\eps(s)z)$. For fixed $s>0$ and $z\in I$, we make a change of variables and find
\begin{align}
 &\lf\|
 e^{-(\Lambda+\eps(a+b+c))s}\,
 \mu\bigl(\tfrac{e^{\eps cs}}{d_\eps(s)}({y}-\beta_\eps(s)z)\bigr)
 \ri\|_{L^2_y}\notag\\
 &\hspace{25mm}
 =d_\eps(s)^{1/2}\,
 e^{-[\Lambda+\eps(a+b+3c/2)]s}\,\|\mu\|_{L^2},
 \label{eq:dilation}
\end{align}
which, since $d_\eps(s)\leq s$, is bounded for $\Lambda\in J$ and all sufficiently small $\eps$ by
\[
 s^{1/2}e^{-(\min J)s/2}\,\|\mu\|_{L^2},
\]
which is integrable on $(0,\infty)$. For every fixed $s>0$, $e^{\eps cs}/d_\eps(s)\to s^{-1}$ and $\beta_\eps(s)\to1$. Since translations and dilations are continuous in $L^2$, the integrand converges in $L^2_y$, uniformly for $(z,\Lambda)\in I\times J$. Using Minkowski's inequality and dominated convergence, we obtain
\[
 \sup_{(z,\Lambda)\in I\times J}
 \|K_{\eps,\Lambda}(\cdot,z)-K_{0,\Lambda}(\cdot,z)\|_{L^2_y}
 \longrightarrow0,
\]
and, since $n$ is bounded and supported in $I$, we also have
\[
 \sup_{\Lambda\in J}
 \int_I\int_\R
 |K_{\eps,\Lambda}(y,z)-K_{0,\Lambda}(y,z)|^2\,
 |n(z)|^2\,\di y\,\di z\longrightarrow0.
\]
This convergence implies \eqref{eq:Klimit}, while applying \eqref{eq:dilation} to each kernel separately gives $K_{\eps,\Lambda}(y,z)n(z)\in {L^2_{y,z}}$. Hence each operator is Hilbert-Schmidt and, in particular, compact. At $\eps=0$, the Fourier transform in $y-z$ is
\[
 \widehat{K_{0,\Lambda}(\cdot,0)}(k)
 =-\int_0^\infty s e^{-\Lambda s}\widehat\mu(ks)\,\di s
 =m_\Lambda(k),
\]
so $\cT_{0,\Lambda}=\cT_\Lambda$.

To prove (iii), we first take $\varphi\in C_c^\infty(\R)$. In \eqref{eq:resolvent}, we change variables in the integral with respect to $w$ using
\[
 z=y_s=e^{\eps bs}\bigl(y-d_\eps(s)w\bigr),\qquad
 w_s=\frac{e^{\eps cs}y-e^{-\eps s}z}{d_\eps(s)},\qquad
 \left|\frac{\pt w}{\pt z}\right|=\frac{e^{-\eps bs}}{d_\eps(s)}.
\]
Since
\[
 \pt_y\bigl[\mu(w_s(y,z))\bigr]
 =\frac{e^{\eps cs}}{d_\eps(s)}\mu'(w_s(y,z)),
\]
we obtain
\be\label{eq:rho H}
 \bal
 \rho_H(y)
 &=-\int_0^\infty
 \frac{e^{-(\Lambda+\eps(a+b))s}}{d_\eps(s)}
 \int_\R\varphi(z)\mu'\!\left(
 \frac{e^{\eps cs}y-e^{-\eps s}z}{d_\eps(s)}\right)\,\di z\,\di s\\
 &=\pt_y\int_\R K_{\eps,\Lambda}(y,z)\varphi(z)\,\di z
 =\pt_y(\cT_{\eps,\Lambda}\varphi)(y)
 \eal
\ee
in the sense of distributions. By Lemma~\ref{lem:resolvent},
\be\label{eq:H L2}
 \|H\|_{{L^2_{y,w}}}
 \leq\frac{\|\varphi\|_{L^2}\|\mu'\|_{L^2}}
 {\Lambda-3\eps/2},
\ee
while taking the $L^1_wL^2_y$ norm in \eqref{eq:resolvent} gives
\be\label{eq:H mixed}
 \|H\|_{L^1_wL^2_y}
 \leq\|\varphi\|_{L^2}\|\mu'\|_{L^1}
 \int_0^\infty e^{-[\Lambda+\eps(a+b/2+c)]s}\,\di s,
\ee
where, after decreasing $\eps$ if necessary, the integral is finite uniformly for $\Lambda\in J$.
For a general compactly supported $\varphi\in L^2$, we then choose smooth approximations supported in a fixed compact interval and use \eqref{eq:H L2} and \eqref{eq:H mixed} to deduce convergence of the reconstructed functions in $L^2\cap L^1_wL^2_y$, and hence of their densities in $L^2_y$. The electric fields also converge in $L^2_y$ by \eqref{eq:dilation}, so we can pass to the limit in \eqref{eq:rho H}.
Finally, if $\varphi=nE$ and $E'=\rho_H$, then
\[
 \pt_y\bigl(E-\cT_{\eps,\Lambda}(nE)\bigr)=0,
\]
so the difference must vanish because it belongs to $L^2(\R)$, while the converse follows directly by differentiation.
\end{proof}

{
The following lemma allows us to retain the simple eigenvalue when we restore the similarity drift.

\begin{lemma}\label{lem:persistence}
Let $T_\eps$, $\eps\geq0$, be compact operators on a complex Hilbert space such that $\|T_\eps-T_0\|\to0$ as $\eps\to0^+$. If $\nu_0\ne0$ is an algebraically simple eigenvalue of $T_0$, then every sufficiently small disk centered at $\nu_0$ contains, for all sufficiently small $\eps$, exactly one eigenvalue $\nu_\eps$ of $T_\eps$. This eigenvalue is algebraically simple and $\nu_\eps\to\nu_0$.
\end{lemma}

\begin{proof}
We choose $0<r<|\nu_0|$ such that the closed disk $|z-\nu_0|\leq r$ contains no other point of $\sigma(T_0)$.
{
For $|z-\nu_0|=r$, we write
\[
 zI-T_\eps
 =\bigl[I-(T_\eps-T_0)(zI-T_0)^{-1}\bigr](zI-T_0).
\]
Since $(zI-T_0)^{-1}$ is uniformly bounded on this circle and $\|T_\eps-T_0\|\to0$, the Neumann series gives, for all sufficiently small $\eps$,
\[
 (zI-T_\eps)^{-1}
 =(zI-T_0)^{-1}\bigl[I-(T_\eps-T_0)(zI-T_0)^{-1}\bigr]^{-1}
 \longrightarrow (zI-T_0)^{-1}
\]
uniformly on $|z-\nu_0|=r$.

}
Thus the Riesz projections (see \cite[Chapter~VII]{Conway}) satisfy
\[
 P_\eps:=\frac{1}{2\pi i}\int_{|z-\nu_0|=r}(zI-T_\eps)^{-1}\,\di z,
 \qquad \|P_\eps-P_0\|\longrightarrow0,
\]
where the circle is oriented counterclockwise. For small $\eps$, we have $\|P_\eps-P_0\|<1$.
If $x\in\operatorname{Ran}P_0$ and $P_\eps x=0$, then
\[
 \|x\|=\|(P_0-P_\eps)x\|\leq\|P_0-P_\eps\|\,\|x\|
 \quad\Longrightarrow\quad x=0,
\]
so $P_\eps:\operatorname{Ran}P_0\to\operatorname{Ran}P_\eps$ is injective.
Exchanging $P_0$ and $P_\eps$ gives the reverse inequality between the dimensions, so
\[
 \operatorname{rank}P_\eps=\operatorname{rank}P_0=1.
\]
Since this rank is the sum of the algebraic multiplicities inside the circle, there is exactly one eigenvalue $\nu_\eps$ there, and it is algebraically simple. Repeating the argument with arbitrarily smaller circles gives $\nu_\eps\to\nu_0$.
\end{proof}

}

\begin{theorem}
\label{thm:drift}
Fix $b=c+1$, $a=-2-c$, $\Lambda_0>0$, and $u>\gamma>0$. There exist $\eps_0>0$ and smooth nonnegative profiles $n_\eps\in C_c^\infty(\R)$, $0<\eps<\eps_0$, converging in $C_c^\infty$ to a nonzero profile $n_*$, such that the linearization of the self-similar Vlasov-Poisson equation around
\be\label{eq:Feps}
 \overline F_\eps(\xi,w)
 :=\eps^{-2}\,n_\eps(\xi/\eps)\,\mu_{u,\gamma}(w)
\ee
has the real unstable eigenvalue
\be\label{eq:lambda}
 \lambda_\eps=\frac{\Lambda_0}{\eps}>0.
\ee
Moreover, an eigenfunction may be chosen real-valued and satisfies
\[
 h_\eps\in {L^2_{\xi,w}}\cap L^1_wL^2_\xi,
 \qquad
 E[\rho_{h_\eps}]\in L^2(\R).
\]
Also, the field operator $E\mapsto\cT_{\eps,\Lambda_0}(n_\eps E)$ has the algebraically simple eigenvalue $1$.
\end{theorem}

\begin{proof}
{We choose $n_*$ as in Theorem~\ref{thm:mode}, so that $1$ is an algebraically simple eigenvalue of $E\mapsto\cT_{0,\Lambda_0}(n_*E)$. By Lemma~\ref{lem:kernels}, the compact operators $\cT_{\eps,\Lambda_0}(n_*\,\cdot)$ converge in norm to $\cT_{0,\Lambda_0}(n_*\,\cdot)$.
Lemma~\ref{lem:persistence} then gives a unique algebraically simple eigenvalue $\nu_\eps$ in a small disk centered at $1$, with $\nu_\eps\to1$. Since the kernels are real, $\overline{\nu_\eps}$ is also an eigenvalue in the same disk. Thus $\nu_\eps$ is real, and its convergence to $1$ ensures positivity for small $\eps$.}

To restore the field eigenvalue to $1$, we set
\[
 n_\eps:=\nu_\eps^{-1}n_*,
\]
and the linear dependence of the field operator on the spatial profile then ensures that $E\mapsto\cT_{\eps,\Lambda_0}(n_\eps E)$ has the simple eigenvalue $1$. Choose a nonzero real eigenvector $E\in L^2(\R)$ and define
\be\label{eq:Heps}
 H_\eps
 :=-(\Lambda_0-\mathcal G_\eps)^{-1}
   \bigl(n_\eps(y)E(y)\mu_{u,\gamma}'(w)\bigr),
\ee
so Lemma~\ref{lem:kernels}(iii), applied with $\varphi=n_\eps E$, gives $H_\eps\in {L^2_{y,w}}\cap L^1_wL^2_y$, $\rho_{H_\eps}\in L^2_y$, and $E'=\rho_{H_\eps}$.
Since $E\in L^2(\R)$, we identify it with the unique $L^2$ primitive $E[\rho_{H_\eps}]$.

We also have $H_\eps\ne0$, since otherwise $E'=\rho_{H_\eps}=0$ would force $E$ to be constant, and hence zero because $E\in L^2(\R)$, contradicting our choice of $E$. Therefore $H_\eps$ is a nonzero solution of
\[
 \bigl(\cL_{n_\eps,\mu_{u,\gamma}}^{\rm loc}
 +\eps\,\mathcal D_{b,c}\bigr)H_\eps
 =\Lambda_0H_\eps,
\]
so, applying Lemma~\ref{lem:debye}, we obtain an eigenfunction $h_\eps=\mathcal R_\eps H_\eps$ of the original self-similar linearization with eigenvalue $\Lambda_0/\eps$. Both $H_\eps$ and $h_\eps$ are real-valued, since $E$ and $n_\eps$ are real-valued and the resolvent in \eqref{eq:resolvent} preserves this property. Moreover, the fixed spatial dilation preserves $L^2\cap L^1_wL^2$, so it only remains to check the electric field, for which \eqref{eq:Eeps} gives
\[
 E[\rho_{h_\eps}](\eps y)=\eps E[\rho_{H_\eps}](y),
\]
from which we deduce that $E[\rho_{h_\eps}]\in L^2(\R)$. Finally, $\nu_\eps\to1$ implies $n_\eps\to n_*$ in $C_c^\infty$.
\end{proof}

\section{Adding a positive tail}
\label{sec:tails}

The background in Theorem~\ref{thm:drift} is compactly supported in space, so a pointwise bound by this background would also force the eigenfunction to have compact spatial support. For the operator without the similarity drift, the following proposition rules out such eigenfunctions.
Motivated by this obstruction, we add a strictly positive integrable tail to bound the eigenfunction along the characteristics and keep the two branches nonnegative, adjusting the coefficient of the compact part so that the rescaled eigenvalue remains equal to $\Lambda_0$.

\begin{prop}
\label{prop:noncompact}
Let $\Lambda\in\C$ with $\re\Lambda>0$, let $n\in C_c^\infty(\R)$, and let $\mu\in C^1(\R)$ be such that $\mu'$ is nonzero on an open interval contained either in $(0,\infty)$ or in $(-\infty,0)$. Suppose that $H\in {L^2_{y,w}}\cap L^1_wL^2_y$ and $E\in L^2(\R)$ solve, in the sense of distributions,
\[
 (\Lambda+w\,\pt_y)H=-n(y)E(y)\mu'(w),
 \qquad E'=\rho_H,
\]
and assume that $H$ is compactly supported in the $y$-variable. Then $H=0$.
\end{prop}

\begin{proof}
The function $nE$ is compactly supported and belongs to $L^1(\R)$ by Cauchy-Schwarz. Assume first that $\mu'$ is nonzero on an interval $I\subset(0,\infty)$. Testing the equation with tensor-product test functions, we find that, for almost every $w\in I$, the one-dimensional equation in $y$ holds in the sense of distributions, and solving it from $y=-\infty$ gives
\be\label{eq:outgoing}
 H(y,w)
 =-\frac{\mu'(w)}{w}
 \int_{-\infty}^y e^{-\Lambda(y-z)/w}\,n(z)E(z)\,\di z.
\ee
Taking $R$ such that $\spt n\subset[-R,R]$, we can then write, for $y>R$,
\[
 H(y,w)
 =-\frac{\mu'(w)}{w}\,e^{-\Lambda y/w}
 \int_\R e^{\Lambda z/w}n(z)E(z)\,\di z,
\]
which vanishes for all sufficiently large $y$ and almost every $w\in I$, because $H$ is compactly supported in $y$. Hence the entire function
\[
 Q(\zeta):=\int_\R e^{\zeta z}n(z)E(z)\,\di z
\]
vanishes at the points $\zeta=\Lambda/w$ for almost every $w\in I$. These points have an accumulation point in $\C$, so the identity theorem implies $Q\equiv0$, and hence $nE=0$.

The transport equation therefore reduces to $(\Lambda+w\,\pt_y)H=0$, whose only solution compactly supported in $y$ is zero for almost every $w\neq0$, so $H=0$. If the interval lies in $(-\infty,0)$, we solve from $y=+\infty$ and use the same Laplace-transform argument.
\end{proof}

Proposition~\ref{prop:noncompact} applies to $\mu_{u,\gamma}$ and shows that the eigenfunction in Theorem~\ref{thm:mode} extends outside the support of $n_*$.

\subsection{General phase-space profiles}

Since the tail will not separate into a spatial and a velocity factor, we extend the rescaling in Lemma~\ref{lem:debye} to a general smooth profile $G=G(y,w)$, for which we define
\be\label{eq:LG}
 \cL_G^{\rm loc}H
 :=-w\,\pt_yH-E[\rho_H]\,\pt_wG.
\ee

\begin{lemma}
\label{lem:scaling}
Let $\eps>0$, $b=c+1$, $a=-2-c$, and $G\in C^\infty(\R^2)\cap L^1(\R^2)$. Set $\overline F_{\eps,G}(\xi,w):=\eps^{-2}G(\xi/\eps,w)$ and choose the ion density in self-similar variables equal to $\rho_{\overline F_{\eps,G}}$. Suppose that
\[
 H\in {L^2_{y,w}}\cap L^1_wL^2_y,
 \qquad
 E[\rho_H]\in L^2(\R),
 \qquad
 E[\rho_H]\,\pt_wG\in {L^2_{y,w}}.
\]
Then, in the sense of distributions,
\be\label{eq:conjugacy G}
 \mathcal R_\eps^{-1}\cL_{\overline F_{\eps,G}}\mathcal R_\eps H
 =\eps^{-1}\cL_G^{\rm loc}H+\mathcal D_{b,c}H.
\ee
\end{lemma}

\begin{proof}
For $h=\mathcal R_\eps H$, \eqref{eq:Eeps} and the definition of $\overline F_{\eps,G}$ give
\[
 E[\rho_h](\eps y)\,\pt_w\overline F_{\eps,G}(\eps y,w)
 =\eps^{-1}E[\rho_H](y)\,\pt_wG(y,w),
\]
and, using \eqref{eq:derivatives} for the transport terms, we obtain \eqref{eq:conjugacy G}.
\end{proof}

To retain a simple eigenvalue near $1$ when we add the tail, we need continuity in operator norm with respect to the profile $G\in L^2$. With $d_\eps(s)=(1-e^{-\eps s})/\eps$, the Birman-Schwinger operator for $G$ is defined by
\be\label{eq:KG}
 \bal
 K_{\eps,\Lambda,G}(y,z)
 &:=-\int_0^\infty e^{-(\Lambda+\eps(a+b+c))s}\,
 G\lf(z,
 \frac{e^{\eps cs}y-e^{-\eps s}z}{d_\eps(s)}
 \ri)\,\di s,\\
 (\cT_{\eps,\Lambda,G}E)(y)
 &:=\int_\R K_{\eps,\Lambda,G}(y,z)\,E(z)\,\di z.
 \eal
\ee
In particular, if $G(y,w)=n(y)\mu(w)$, then
\[
 \cT_{\eps,\Lambda,G}E=\cT_{\eps,\Lambda}(nE).
\]

\begin{lemma}
\label{lem:KG}
Let $\eps>0$, $b=c+1$, $a=-2-c$, and $\Lambda\in\R$. Suppose that $\Lambda>3\eps/2$ and
\be\label{eq:exponents}
 \Lambda+\eps(a+b+3c/2)>0,
 \qquad
 \Lambda+\eps(a+b/2+c)>0.
\ee
The following properties hold.
\begin{itemize}
\item[(i)] The map $G\mapsto\cT_{\eps,\Lambda,G}$ extends linearly and continuously from ${L^2_{y,w}}$ to the Hilbert-Schmidt operators on $L^2(\R)$.
\item[(ii)] Let $G\in C^\infty(\R^2)\cap {L^2_{y,w}}$ and suppose that $E\in L^2(\R)$ is a fixed point of $\cT_{\eps,\Lambda,G}$ such that
\be\label{eq:EG}
 E\,\pt_wG\in {L^2_{y,w}}\cap L^1_wL^2_y.
\ee
Then the function
\be\label{eq:HG}
 H(y,w)=-\int_0^\infty e^{-(\Lambda+\eps a)s}\,
 E(y_s)\,\pt_wG(y_s,w_s)\,\di s
\ee
belongs to ${L^2_{y,w}}\cap L^1_wL^2_y$, and the pair $(H,E)$ solves, in the sense of distributions,
\be\label{eq:mode G}
 (\cL_G^{\rm loc}+\eps\,\mathcal D_{b,c})H=\Lambda H,
 \qquad E'=\rho_H.
\ee
\item[(iii)] Conversely, let $G\in C^\infty(\R^2)\cap {L^2_{y,w}}$. {Suppose that $(H,E)$ satisfies \eqref{eq:mode G} in the sense of distributions, with $H,E\,\pt_wG\in L^2_{y,w}\cap L^1_wL^2_y$ and $E\in L^2(\R)$.} Then $E$ is a fixed point of $\cT_{\eps,\Lambda,G}$.
\end{itemize}
\end{lemma}

\begin{proof}
To prove (i), we fix $s>0$ and $z$ and change variables in the second argument of $G$ to get
\[
 \lf\|G\lf(z,
 \frac{e^{\eps cs}\,\cdot-e^{-\eps s}z}{d_\eps(s)}
 \ri)\ri\|_{L^2_y}
 =d_\eps(s)^{1/2}e^{-\eps cs/2}\,
 \|G(z,\cdot)\|_{L^2_w},
\]
and Minkowski's inequality then gives
\be\label{eq:KG norm}
 \|K_{\eps,\Lambda,G}\|_{L^2_{y,z}}
 \leq \|G\|_{L^2_{y,w}}
 \int_0^\infty d_\eps(s)^{1/2}\,
 e^{-[\Lambda+\eps(a+b+3c/2)]s}\,\di s.
\ee
Since $0\leq d_\eps(s)\leq\eps^{-1}$, the last integral is finite by \eqref{eq:exponents}. Together with the linear dependence of the kernel on $G$, this proves (i).

For (ii), changing variables from $(y,w)$ to $(y_s,w_s)$, we obtain
\[
 \bal
 \|(E\pt_wG)(y_s,w_s)\|_{L^2_{y,w}}
 &=e^{-\eps(b+c)s/2}\|E\pt_wG\|_{L^2_{y,w}},\\
 \|(E\pt_wG)(y_s,w_s)\|_{L^1_wL^2_y}
 &=e^{-\eps(b/2+c)s}\|E\pt_wG\|_{L^1_wL^2_y},
 \eal
\]
so, applying Minkowski's inequality in \eqref{eq:HG} and using $a+(b+c)/2=-3/2$, we deduce
\be\label{eq:HG norm}
 \|H\|_{L^2}
 \leq\frac{\|E\pt_wG\|_{L^2}}{\Lambda-3\eps/2},
 \qquad
 \|H\|_{L^1_wL^2_y}
 \leq
 \frac{\|E\,\pt_wG\|_{L^1_wL^2_y}}
 {\Lambda+\eps(a+b/2+c)}.
\ee
In particular, we have $\rho_H\in L^2$, while using the resolvent formula \eqref{eq:resolvent} gives
\[
 (\Lambda-\mathcal G_\eps)H=-E\pt_wG.
\]
{To compute $\rho_H$, we use the change of variables from the proof of Lemma~\ref{lem:kernels}(iii), with $E(z)G(z,w_s(y,z))$ in place of $\varphi(z)\mu(w_s(y,z))$.}
Using
\[
 \pt_y\bigl[G(z,w_s(y,z))\bigr]
 =\frac{e^{\eps cs}}{d_\eps(s)}\,\pt_wG(z,w_s(y,z)),
\]
we obtain
\be\label{eq:rho G}
 \rho_H(y)=\pt_y\int_\R K_{\eps,\Lambda,G}(y,z)E(z)\,\di z
       =\pt_y\bigl(\cT_{\eps,\Lambda,G}E\bigr)(y)
\ee
in the sense of distributions. We justify this identity by first integrating over $\delta<s<R$, testing against a function in $C_c^\infty(\R_y)$, and then letting $\delta\to0^+$ and $R\to\infty$, using \eqref{eq:KG norm} and the mixed-norm estimate in \eqref{eq:HG norm}. Since $E=\cT_{\eps,\Lambda,G}E$, we conclude that $E'=\rho_H$.

For (iii), we apply the $L^2$ resolvent to the first equation in \eqref{eq:mode G}, so uniqueness of the inverse gives $H$ in the form \eqref{eq:HG}. Equation~\eqref{eq:rho G} then shows that $(E-\cT_{\eps,\Lambda,G}E)'=0$, and since both terms belong to $L^2(\R)$, the constant $E-\cT_{\eps,\Lambda,G}E$ must be zero.
\end{proof}

\subsection{A weight adapted to the characteristics}

We write $\langle z\rangle=(1+|z|^2)^{1/2}$ and assume $c>0$. We choose the positive tail so that it also bounds the eigenfunction reconstructed along the characteristics. Since the velocity derivative of the compact background contains $\mu_{u,\gamma}'$, which decays like $\langle w\rangle^{-5}$, we take the same decay in velocity for the tail. For the spatial factor, we use the identity
\be\label{eq:ray}
 \eps y_s-w_s=e^{\eps bs}(\eps y-w),
\ee
which follows from \eqref{eq:flow}. For $\eps>0$ and an integer $N>2$, we define

\be\label{eq:weight}
 Z_\eps(y,w):=\frac{\eps y-w}{\langle w\rangle},
 \qquad
 \Omega_\eps(y,w):=\langle w\rangle^{-5}\,\langle Z_\eps(y,w)\rangle^{-N},
\ee
where the factor $\langle Z_\eps\rangle^{-N}$ controls the decay away from the line $w=\eps y$ on the velocity scale $\langle w\rangle$.

\begin{lemma}
\label{lem:weight}
Let $0<\eps\leq1$, let $N>2$ be an integer, and let $b=c+1$ with $c>0$.
The following properties hold.
\begin{itemize}
\item[(i)] The function $\Omega_\eps$ is smooth and strictly positive.
\item[(ii)] We have
\be\label{eq:Omega moments}
 {\iint_{\R^2}(1+|y|+|w|)\,\Omega_\eps(y,w)\,\di y\,\di w}<\infty,
 \qquad \Omega_\eps\in {L^2_{y,w}}.
\ee
\item[(iii)] For a constant $C$ depending only on $N$, we have
\be\label{eq:Omega derivatives}
 \bal
 |\pt_w\Omega_\eps|&\leq C\langle w\rangle^{-1}\Omega_\eps,&
 |\pt_{ww}\Omega_\eps|&\leq C\langle w\rangle^{-2}\Omega_\eps,\\
 |\pt_{yw}\Omega_\eps|&\leq C\eps\langle w\rangle^{-2}\Omega_\eps,&
 |w\,\pt_y\Omega_\eps|+|y\,\pt_y\Omega_\eps|
 +|w\,\pt_w\Omega_\eps|&\leq C\,\Omega_\eps.
 \eal
\ee
In particular, $\pt_{yw}\Omega_\eps,\pt_{ww}\Omega_\eps\in L^1(\R^2)$.
\item[(iv)] The function
\be\label{eq:Aeps}
 A_\eps(y):=\int_\R|\pt_w\Omega_\eps(y,w)|\,\di w
\ee
belongs to $L^2(\R)$.
\item[(v)] Along the flow \eqref{eq:flow}, we have
\be\label{eq:Omega flow}
 \Omega_\eps(y_s,w_s)\leq\Omega_\eps(y,w)
 \qquad\text{for every }s\geq0.
\ee
\end{itemize}
\end{lemma}

\begin{proof}
The smoothness and strict positivity in (i) follow directly from the definition \eqref{eq:weight}.
To prove (ii), we fix $w$ and use $Z:=Z_\eps(y,w)$ as the spatial variable, so that $\di y=\langle w\rangle\eps^{-1}\di Z$ and
\[
 \|\Omega_\eps\|_{L^1}
 =\frac1\eps\int_\R \langle w\rangle^{-4}\,\di w
 \int_\R\langle Z\rangle^{-N}\,\di Z,
\]
while
\[
 \|\Omega_\eps\|_{L^2}^2
 =\frac1\eps\int_\R \langle w\rangle^{-9}\,\di w
  \int_\R\langle Z\rangle^{-2N}\,\di Z.
\]
For the first moments, we have
\begin{align*}
 {\iint_{\R^2}|w|\Omega_\eps\,\di y\,\di w}
 &=\frac1\eps\int_\R |w|\langle w\rangle^{-4}\,\di w
   \int_\R\langle Z\rangle^{-N}\,\di Z,\\
 {\iint_{\R^2}|y|\Omega_\eps\,\di y\,\di w}
 &\leq\frac1{\eps^2}\lf(
   \int_\R \langle w\rangle^{-3}\,\di w\int_\R|Z|\langle Z\rangle^{-N}\,\di Z
   +\int_\R|w|\langle w\rangle^{-4}\,\di w\int_\R\langle Z\rangle^{-N}\,\di Z
   \ri),
\end{align*}
and these integrals are finite because $N>2$, which proves (ii).
For (iii), we differentiate and find

\be\label{eq:Z derivatives}
 \bal
 \pt_yZ_\eps&=\eps\langle w\rangle^{-1},
 &|\pt_{yw}Z_\eps|&\leq C\eps\langle w\rangle^{-2},\\
 \pt_wZ_\eps&=-\langle w\rangle^{-1}-w\langle w\rangle^{-2}Z_\eps,
 &|\pt_{ww}Z_\eps|&\leq C\langle w\rangle^{-2}(1+|Z_\eps|).
 \eal
\ee
The first derivatives satisfy
\be\label{eq:Omega gradient}
 \frac{\pt_y\Omega_\eps}{\Omega_\eps}
 =-\frac{N\eps Z_\eps}{\langle w\rangle\langle Z_\eps\rangle^2},
 \qquad
 \frac{\pt_w\Omega_\eps}{\Omega_\eps}
 =-\frac{5w}{\langle w\rangle^2}
   -\frac{NZ_\eps\pt_wZ_\eps}{\langle Z_\eps\rangle^2},
\ee
so, using \eqref{eq:Z derivatives}, $0<\eps\leq1$, and $\eps y=\langle w\rangle Z_\eps+w$, we deduce the first-order bounds in \eqref{eq:Omega derivatives}.
For the second derivatives, we use
\[
 \frac{\pt_{iw}\Omega_\eps}{\Omega_\eps}
 =\pt_w\lf(\frac{\pt_i\Omega_\eps}{\Omega_\eps}\ri)
  +\frac{\pt_i\Omega_\eps}{\Omega_\eps}
   \frac{\pt_w\Omega_\eps}{\Omega_\eps},
 \qquad i\in\{y,w\},
\]
and, differentiating the quotients in \eqref{eq:Omega gradient} and using the first-order bounds in \eqref{eq:Omega derivatives}, we obtain
\begin{align*}
 \frac{|\pt_{ww}\Omega_\eps|}{\Omega_\eps}
 &\leq C\lf(
  \langle w\rangle^{-2}
  +\frac{|\pt_wZ_\eps|^2}{\langle Z_\eps\rangle^2}
  +\frac{|\pt_{ww}Z_\eps|}{\langle Z_\eps\rangle}
 \ri)
 \leq C\langle w\rangle^{-2},\\
 \frac{|\pt_{yw}\Omega_\eps|}{\Omega_\eps}
 &\leq C\eps\lf(
  \frac{1}{\langle w\rangle^2\langle Z_\eps\rangle}
  +\frac{|\pt_wZ_\eps|}{\langle w\rangle\langle Z_\eps\rangle^2}
 \ri)
 \leq C\eps\langle w\rangle^{-2}.
\end{align*}

Since $\Omega_\eps\in L^1(\R^2)$, these bounds also show that {$\pt_{yw}\Omega_\eps,\pt_{ww}\Omega_\eps\in L^1(\R^2)$}, completing the proof of (iii).
We note in addition that the first bound in \eqref{eq:Omega derivatives} implies $\sup_y|\pt_w\Omega_\eps(y,w)|\leq C\langle w\rangle^{-6}$.
To prove (iv), we take the $L^2_y$ norm in the first bound of \eqref{eq:Omega derivatives}, use $Z=Z_\eps(y,w)$ as the spatial variable, and apply Minkowski's inequality to conclude that
\[
 \|\pt_w\Omega_\eps(\cdot,w)\|_{L^2_y}
 \leq C\eps^{-1/2}\langle w\rangle^{-11/2},
 \qquad
 \|A_\eps\|_{L^2_y}
 \leq\int_\R
 \|\pt_w\Omega_\eps(\cdot,w)\|_{L^2_y}\,\di w<\infty.
\]

For (v), we use \eqref{eq:ray} and $\langle w\rangle\leq\langle w_s\rangle\leq e^{\eps cs}\langle w\rangle$ to obtain
\[
 |Z_\eps(y_s,w_s)|
 =e^{\eps bs}\,\frac{\langle w\rangle}{\langle w_s\rangle}\,|Z_\eps(y,w)|
 \geq e^{\eps(b-c)s}|Z_\eps(y,w)|
 =e^{\eps s}|Z_\eps(y,w)|,
\]
so, by the flow property, both $\langle w_s\rangle^{-5}$ and $\langle Z_\eps(y_s,w_s)\rangle^{-N}$ are nonincreasing in $s$, and their product satisfies (v).
\end{proof}

\subsection{{Construction of the positive background}}

The pointwise bound by the background will keep the two branches nonnegative, while control of the field in $L^\infty$ and of the velocity derivative in $L^1$ will make the quadratic term $E[\rho_H]\pt_wH$ in their common source integrable.

\begin{theorem}
\label{thm:tail}
Fix $b>2$ and an integer $N>2$ in \eqref{eq:weight}, set $c=b-1$ and $a=-b-1$, and fix $\Lambda_0>0$ and $u>\gamma>0$. Let $n_\eps$ be the profiles given by Theorem~\ref{thm:drift}. For all sufficiently small $\eps>0$ there exist $\theta_0>0$ and, for every $0<\theta<\theta_0$, a number $\alpha_{\eps,\theta}>0$ such that, setting
\be\label{eq:Gtheta}
 G_{\eps,\theta}(y,w)
 :=\alpha_{\eps,\theta}\,n_\eps(y)\,\mu_{u,\gamma}(w)
 +\theta\,\Omega_\eps(y,w),
\ee
there exists a nonzero real function
\[
 H_{\eps,\theta}
 \in {L^2_{y,w}}\cap L^1_wL^2_y\cap L^1(\R^2)
\]
such that, in the sense of distributions,
\[
 (\cL_{G_{\eps,\theta}}^{\rm loc}+\eps\mathcal D_{b,c})H_{\eps,\theta}
 =\Lambda_0H_{\eps,\theta}.
\]
{We have $E[\rho_{H_{\eps,\theta}}]\in H^1(\R)$}, and
\be\label{eq:pointwise}
 |H_{\eps,\theta}(y,w)|\leq C_{\eps,\theta}\,\Omega_\eps(y,w),
 \qquad
 \pt_wH_{\eps,\theta}\in L^1(\R^2).
\ee
In particular, the self-similar background
\be\label{eq:Ftheta}
 \overline F_{\eps,\theta}(\xi,w)
 :=\eps^{-2}\,G_{\eps,\theta}(\xi/\eps,w)
\ee
is smooth, strictly positive, and integrable, and its linearization has the eigenvalue $\Lambda_0/\eps$ with eigenfunction $h_{\eps,\theta}=\mathcal R_\eps H_{\eps,\theta}$. This eigenfunction satisfies a pointwise bound by a constant multiple of $\overline F_{\eps,\theta}$.
\end{theorem}

\begin{proof}
\step{1}{}
Fix $\eps>0$ sufficiently small so that
\be\label{eq:parameters}
 \Lambda_0>(b+1)\eps,
\ee
so the self-similar growth rate $\Lambda_0/\eps$ exceeds $b+1=-a$, ensuring convergence of the characteristic integral used for the pointwise bound in Step~3. Since $b>2$ and $c=b-1$, this choice also implies the parameter conditions in Lemma~\ref{lem:KG} with $\Lambda=\Lambda_0$.
Setting $G_0:=n_\eps\mu_{u,\gamma}$, we know from Theorem~\ref{thm:drift} that the compact operator $\cT_{\eps,\Lambda_0,G_0}$ has the algebraically simple eigenvalue $1$.
By the linearity in $G$ proved in Lemma~\ref{lem:KG},
\[
 \cT_{\eps,\Lambda_0,\alpha G_0+\theta\Omega_\eps}
 =\alpha\,\cT_{\eps,\Lambda_0,G_0}
 +\theta\,\cT_{\eps,\Lambda_0,\Omega_\eps}.
\]
The simple eigenvalue $1$ therefore extends to a real analytic branch $\nu(\alpha,\theta)$ near $(1,0)$, with $\nu(1,0)=1$ and real values for real $(\alpha,\theta)$ {(see \cite[Chapter~II, \S5.7, and Chapter~VII, \S1.3]{Kato})}.
At $\theta=0$, we have
\[
 \nu(\alpha,0)=\alpha,
 \qquad \pt_\alpha\nu(1,0)=1,
\]
so the implicit function theorem provides a real analytic function $\theta\mapsto\alpha_{\eps,\theta}$ such that $\alpha_{\eps,0}=1$ and $\nu(\alpha_{\eps,\theta},\theta)=1$, with $\alpha_{\eps,\theta}>0$ for $\theta$ sufficiently small by continuity.

\step{2}{}
Let $E$ be a real eigenvector of $\cT_{\eps,\Lambda_0,G_{\eps,\theta}}$ with eigenvalue $1$, normalized by $\|E\|_{L^2}=1$, and reconstruct $H$ by \eqref{eq:HG}.
To apply Lemma~\ref{lem:KG}(ii), we estimate the two terms in $E\,\pt_wG_{\eps,\theta}$ separately.
For the compact part,
\[
 \|n_\eps E\mu_{u,\gamma}'\|_{L^2_{y,w}}
 =\|n_\eps E\|_{L^2_y}\|\mu_{u,\gamma}'\|_{L^2_w},
 \qquad
 \|n_\eps E\mu_{u,\gamma}'\|_{L^1_wL^2_y}
 =\|n_\eps E\|_{L^2_y}\|\mu_{u,\gamma}'\|_{L^1_w}.
\]
For the tail, we use the bound $\sup_y|\pt_w\Omega_\eps(y,w)|\leq C\langle w\rangle^{-6}$ to estimate
\begin{align*}
 \|E\,\pt_w\Omega_\eps\|_{L^1_wL^2_y}
 &\leq \|E\|_{L^2_y}
 \int_\R\sup_y|\pt_w\Omega_\eps(y,w)|\,\di w,\\
 \|E\,\pt_w\Omega_\eps\|_{L^2_{y,w}}^2
 &\leq \|E\|_{L^2_y}^2
 \int_\R\sup_y|\pt_w\Omega_\eps(y,w)|^2\,\di w.
\end{align*}
Combining these estimates, we obtain
\[
 E\,\pt_wG_{\eps,\theta}
 \in {L^2_{y,w}}\cap L^1_wL^2_y,
\]
so we may apply Lemma~\ref{lem:KG}(ii) and \eqref{eq:HG norm} to deduce $\rho_H\in L^2$ and $E'=\rho_H$, hence $E\in H^1(\R)\subset L^\infty(\R)$.
{Moreover, if $H=0$, then $E'=\rho_H=0$, so $E$ is constant and, since $E\in L^2(\R)$, we have $E=0$. This contradicts $\|E\|_{L^2}=1$, and therefore $H\ne0$.}
As $E$ and $G_{\eps,\theta}$ are real, we also see from \eqref{eq:HG} that $H$ is real.

\step{3}{}
{To prove the pointwise bound in \eqref{eq:pointwise}, we use \eqref{eq:HG} and observe that the compact part contributes only when $y_s$ lies in the support of $n_\eps$.} Taking $R$ such that $\spt n_\eps\subset[-R,R]$, \eqref{eq:ray} and \eqref{eq:flow} give, whenever $n_\eps(y_s)\ne0$,
\[
 |Z_\eps(y,w)|
 \leq e^{-\eps bs}\,
 \frac{\eps R+|w_s|}{\langle w\rangle}
 \leq\eps R+1,
\]
and, since $n_\eps$ is bounded and $|\mu_{u,\gamma}'(w_s)|\leq C\langle w_s\rangle^{-5}$, we obtain on these characteristics
\[
 |n_\eps(y_s)\mu_{u,\gamma}'(w_s)|
 \leq C\langle w\rangle^{-5}
 \leq C\Omega_\eps(y,w).
\]
For the tail, Lemma~\ref{lem:weight} gives
\[
 |\pt_w\Omega_\eps(y_s,w_s)|
 \leq C\langle w_s\rangle^{-1}\Omega_\eps(y_s,w_s)
 \leq C\,\Omega_\eps(y,w).
\]
Using these estimates, $E\in L^\infty$, and \eqref{eq:parameters} in \eqref{eq:HG}, we find
\[
 |H(y,w)|
 \leq\frac{C\,\|E\|_{L^\infty}}
 {\Lambda_0+\eps a}\,\Omega_\eps(y,w),
\]
and, since $\Omega_\eps\in L^1(\R^2)$, we also have $H\in L^1(\R^2)$ and $\rho_H\in L^1(\R)$.
{
Since $E\in H^1(\R)$, we have $E(y)\to0$ as $y\to\pm\infty$, and hence
\[
 \int_\R\rho_H(y)\,\di y
 =\lim_{R\to\infty}\bigl(E(R)-E(-R)\bigr)=0.
\]
Therefore, by \eqref{eq:field},
\[
 E(y)=\int_{-\infty}^y\rho_H(z)\,\di z=E[\rho_H](y).
\]

}

\step{4}{}
We set $\mu:=\mu_{u,\gamma}$ and write $H=H_{\rm core}+\theta H_{\rm tail}$ according to \eqref{eq:Gtheta}. Since $E\in H^1$ and $n_\eps$ is smooth and compactly supported, $n_\eps E,(n_\eps E)'\in L^1$. Absorbing the fixed factor $\alpha_{\eps,\theta}$ into the constant, we differentiate the characteristic formula \eqref{eq:HG} for $H_{\rm core}$ using
\[
 \pt_wy_s=-e^{\eps bs}d_\eps(s),
 \qquad
 \pt_ww_s=e^{\eps cs},
\]
and then change variables from $(y,w)$ to $(y_s,w_s)$ to obtain
\begin{align}
 \|\pt_wH_{\rm core}\|_{L^1}
 &\leq C\|(n_\eps E)'\|_{L^1}\|\mu'\|_{L^1}
 \int_0^\infty e^{-(\Lambda_0-2\eps)s}d_\eps(s)\,\di s\notag\\*
 &\quad+C\|n_\eps E\|_{L^1}\|\mu''\|_{L^1}
 \int_0^\infty e^{-(\Lambda_0-\eps)s}\,\di s<\infty.
 \label{eq:core}
\end{align}
For $H_{\rm tail}$, the chain rule gives
\[
 \bal
 \pt_w\bigl[E(y_s)\pt_w\Omega_\eps(y_s,w_s)\bigr]
 &=-e^{\eps bs}d_\eps(s)
 \bigl(E'\pt_w\Omega_\eps+E\pt_{yw}\Omega_\eps\bigr)(y_s,w_s)\\
 &\quad+e^{\eps cs}\bigl(E\pt_{ww}\Omega_\eps\bigr)(y_s,w_s),
 \eal
\]
so, using $A_\eps$ from \eqref{eq:Aeps} and the same change of variables, we deduce
\begin{align}
 \|\pt_wH_{\rm tail}\|_{L^1}
 &\leq C\int_0^\infty e^{-(\Lambda_0-2\eps)s}d_\eps(s)\,\di s
 \lf(
 \|E'\|_{L^2}\|A_\eps\|_{L^2}
 +\|E\|_{L^\infty}\|\pt_{yw}\Omega_\eps\|_{L^1}
 \ri)\notag\\
 &\quad+C\int_0^\infty e^{-(\Lambda_0-\eps)s}\,\di s\,
 \|E\|_{L^\infty}\|\pt_{ww}\Omega_\eps\|_{L^1}<\infty.
 \label{eq:tail}
\end{align}
{
Let $H_T$ be defined by \eqref{eq:HG} with the integral taken over $[0,T]$. By the estimates in Step~3 and in \eqref{eq:core} and \eqref{eq:tail}, for $T'>T$ we have
\[
 \begin{aligned}
 \|H-H_T\|_{L^1}&\leq Ce^{-(\Lambda_0+\eps a)T},\\
 \|\pt_wH_{T'}-\pt_wH_T\|_{L^1}
 &\leq C\int_T^{T'}\left(e^{-(\Lambda_0-2\eps)s}d_\eps(s)
 +e^{-(\Lambda_0-\eps)s}\right)\,\di s.
 \end{aligned}
\]
Both right-hand sides tend to zero as $T\to\infty$, so $H_T$ and $\pt_wH_T$ converge in $L^1(\R^2)$. Testing against functions in $C_c^\infty(\R^2)$ identifies the second limit with $\pt_wH$.

}

Having proved \eqref{eq:pointwise}, we return to the self-similar variables $(\xi,w)$ using \eqref{eq:conjugacy G}. The smoothness, strict positivity, and integrability of $G_{\eps,\theta}$ pass to the rescaled background, while $G_{\eps,\theta}\geq\theta\Omega_\eps$ gives the pointwise bound
\[
 |h_{\eps,\theta}(\xi,w)|
 \leq \frac{C_{\eps,\theta}\eps^2}{\theta}
 \overline F_{\eps,\theta}(\xi,w).
\]
\end{proof}

{
\begin{remark}\label{rem:continuity}
The eigenfunction $H=H_{\eps,\theta}$ constructed in Theorem~\ref{thm:tail} has a continuous representative, and $\rho_H$ is continuous and vanishes at infinity.
Indeed, since $E[\rho_H]\in H^1(\R)$ is continuous and $G_{\eps,\theta}$ is smooth, the integrand in \eqref{eq:HG} is continuous in $(y,w)$, while the estimates in Step~3 of the proof bound it by $Ce^{-(\Lambda_0+\eps a)s}$, with $\Lambda_0+\eps a>0$. Thus $H$ is continuous by dominated convergence.
Moreover, $|H(y,w)|\leq C\Omega_\eps(y,w)\leq C\langle w\rangle^{-5}$ and $\Omega_\eps(y,w)\to0$ as $|y|\to\infty$ for every fixed $w$, so the same argument shows that $\rho_H$ is continuous and vanishes at infinity.
\end{remark}
}

\section{Proof of Theorem~\ref{thm:main}}
\label{sec:proof}

To produce two nonnegative solutions with the same source, we take opposite perturbations of the background constructed in Section~\ref{sec:tails}, using the following form of Golovkin's identity \cite{Gol,DM25}. We write
\[
 B(F_1,F_2):=-E[\rho_{F_1}]\,\pt_wF_2.
\]

\begin{lemma}
\label{lem:Golovkin}
Let $\overline F\in C^\infty(\R^2)\cap L^1(\R^2)$ be nonnegative, fix the ion density $\rho_{\overline F}$, and assume that $\mathcal A_{\overline F}(\overline F)\in L^1(\R^2)$.
Let $\lambda>0$ and let $h$ be real-valued, with
\[
 h,\pt_wh\in L^1(\R^2),
 \qquad \cL_{\overline F}h=\lambda h
\]
in the sense of distributions. For $\delta>0$, set $r(\tau):=\delta e^{\lambda\tau}h$.
Then, for every $\tau\leq0$, we have
\be\label{eq:golovkin}
 \pt_\tau(\overline F\pm r)
 -\mathcal A_{\overline F}(\overline F\pm r)
 =-\mathcal A_{\overline F}(\overline F)-B(r,r).
\ee
In particular, both branches have the same additive source in $L^1(\R^2)$. If $|r(\tau)|\leq\overline F$ almost everywhere for every $\tau\leq0$, they are nonnegative.
\end{lemma}

\begin{proof}
Since the ion density is fixed, we have
\[
 \bal
 \rho_{\overline F\pm r}-\rho_{\overline F}&=\pm\rho_r,\\
 -E[\rho_{\overline F\pm r}-\rho_{\overline F}]\,\pt_w(\overline F\pm r)
 &=\mp E[\rho_r]\,\pt_w\overline F-E[\rho_r]\,\pt_wr,
 \eal
\]
so, using \eqref{eq:AF} and \eqref{eq:linearization}, we obtain
\[
 \bal
 \mathcal A_{\overline F}(\overline F\pm r)
 &=\mathcal A_{\overline F}(\overline F)\pm\cL_{\overline F}r+B(r,r),\\
 \pt_\tau r&=\lambda r=\cL_{\overline F}r,
 \eal
\]
and subtracting the first identity from $\pt_\tau(\overline F\pm r)=\pm\pt_\tau r$ gives \eqref{eq:golovkin}. By the definition \eqref{eq:field}, we have
\be\label{eq:quadratic}
 \|B(h,h)\|_{L^1}
 \leq\|E[\rho_h]\|_{L^\infty}\|\pt_wh\|_{L^1}
 \leq\frac12\|h\|_{L^1}\|\pt_wh\|_{L^1}<\infty,
\ee
which, together with $\mathcal A_{\overline F}(\overline F)\in L^1$, proves that the source belongs to $L^1(\R^2)$.
\end{proof}

\proofof{Theorem~\ref{thm:main}}
For a fixed finite $p\in[1,\infty)$, we choose $b>\max\{3,2p\}$ so that the source will vanish in $L^1$ at the initial time and the spatial densities will remain bounded in $L^p$.
We then set $c=b-1$ and $a=-b-1$, choose the remaining parameters as in Theorem~\ref{thm:tail}, and write
\[
 G:=G_{\eps,\theta},\qquad
 \overline F(\xi,w):=\eps^{-2}G(\xi/\eps,w),
 \qquad
 h(\xi,w):=H_{\eps,\theta}(\xi/\eps,w).
\]
In self-similar variables, the electric field of $\overline F$ vanishes by the choice of the ion density, and the linearized operator around $\overline F$ has the eigenvalue $\lambda=\Lambda_0/\eps$ with real eigenfunction $h$.

\step{1}{}
By Theorem~\ref{thm:tail}, there exists $C_h>0$ such that
\[
 |h(\xi,w)|\leq C_h\,\Omega_\eps(\xi/\eps,w),
 \qquad
 \overline F(\xi,w)
 \geq\eps^{-2}\theta\,\Omega_\eps(\xi/\eps,w),
\]
so we choose $\delta>0$ such that $\delta C_h\leq\eps^{-2}\theta$ and, for $\tau\leq0$, set
\be\label{eq:branches}
 r(\tau):=\delta e^{\lambda\tau}h,
 \qquad
 F_\pm(\tau):=\overline F\pm r(\tau).
\ee
Since $\lambda>0$ and $\tau\leq0$, this choice ensures that $|r(\tau)|\leq\overline F$, so both branches are nonnegative. Applying Lemma~\ref{lem:Golovkin}, we obtain the common self-similar source
\be\label{eq:source}
 \mathscr S(\tau)
 :=-\mathcal A_{\overline F}(\overline F)
 +\delta^2e^{2\lambda\tau}E[\rho_h]\,\pt_wh.
\ee
We verify the integrability assumptions of Lemma~\ref{lem:Golovkin} in Step~3 below.

\step{2}{}
We write $H:=H_{\eps,\theta}$ and recall that Step~3 of the proof of Theorem~\ref{thm:tail} establishes the normalization of the field and the identity
\be\label{eq:neutrality}
 \iint_{\R^2}H(y,w)\,\di y\,\di w=0,
\ee
and the fixed spatial dilation gives
\[
 \rho_h(\xi)=\rho_H(\xi/\eps),
 \qquad
 E[\rho_h](\xi)=\eps E[\rho_H](\xi/\eps).
\]

\step{3}{}
Using \eqref{eq:moments} for the compact part of $G$ and \eqref{eq:Omega moments} and \eqref{eq:Omega derivatives} for the tail, we find
\[
 G,\quad w\,\pt_yG,\quad y\,\pt_yG,\quad
 w\,\pt_wG\in L^1(\R^2),
\]
and since the field vanishes at $\overline F$, these estimates imply, after the spatial dilation by $\eps$, that $\mathcal A_{\overline F}(\overline F)\in L^1(\R^2)$.
Theorem~\ref{thm:tail} also ensures that $h,\pt_wh\in L^1$, so, using the quadratic estimate \eqref{eq:quadratic} in \eqref{eq:source}, we obtain
\be\label{eq:S norm}
 \|\mathscr S(\tau)\|_{L^1}
 \leq C_0+C_1e^{2\lambda\tau},
 \qquad \tau\leq0.
\ee

\step{4}{}
For $0<t\leq1$, we define
\be\label{eq:fpm}
 \bal
 f_\pm(t,x,v)
 &:=t^aF_\pm(\log t,xt^{-b},vt^{-c}),\\
 \r_i(t,x)
 &:=t^{-2}\rho_{\overline F}(xt^{-b}),\\
 S(t,x,v)
 &:=t^{a-1}\mathscr S(\log t,xt^{-b},vt^{-c}).
 \eal
\ee
The compact part of $G$ has a smooth spatial density, while for the tail we have, for every integer $j\geq0$,
\[
 |\pt_y^j\Omega_\eps(y,w)|\leq C_j\eps^j\langle w\rangle^{-5-j},
\]
so we can differentiate under the velocity integral and deduce that $\rho_{\overline F}$ is smooth, while its nonnegativity follows from $\overline F\geq0$; therefore $\r_i\in C^\infty((0,1]\times\R)$ and $\r_i\geq0$.
The electric fields of the two branches are
\be\label{eq:Epm}
 E_\pm(t,x)=\pm\,\delta\, t^{b-2+\lambda}E[\rho_h](xt^{-b}).
\ee
Indeed,
\[
 \pt_xE_\pm(t,x)
 =\pm\,\delta\, t^{\lambda-2}\rho_h(xt^{-b})
 =\r_{f_\pm}(t,x)-\r_i(t,x),
\]
and, integrating in $x$ and using \eqref{eq:neutrality}, we find
\[
 \int_\R\bigl(\rho_{f_\pm}(t,x)-\rho_i(t,x)\bigr)\,\di x
 =\pm\delta\eps\,t^{b-2+\lambda}\int_\R\rho_H(y)\,\di y=0.
\]
{By \eqref{eq:field}, we also have $E_\pm(t,x)\to0$ as $x\to\pm\infty$.} Using the scaling relations \eqref{eq:scaling} and the common-source identity \eqref{eq:golovkin}, we conclude that both $f_+$ and $f_-$ solve \eqref{eq:vp} with the same $\r_i$ and $S$, and they are distinct because $h\ne0$.

\step{5}{}
Changing variables in phase space, we obtain
\be\label{eq:mass}
 \|f_\pm(t)\|_{L^1}
 \leq t^{b-2}\lf(
 \eps^{-1}\|G\|_{L^1}
 +\delta\eps\, t^\lambda\|H_{\eps,\theta}\|_{L^1}
 \ri).
\ee
Since $\tau\mapsto F_\pm(\tau)$ is continuous in $L^1$ and the dilation group is strongly continuous on $L^1(\R^2)$, we deduce that $t\mapsto f_\pm(t)$ is continuous on $(0,1]$, while $b>2$ and \eqref{eq:mass} imply $\|f_\pm(t)\|_{L^1}\to0$ as $t\to0^+$, proving the continuity and initial condition in \eqref{eq:initial}. To treat the source, we use the same change of variables together with \eqref{eq:S norm} and find
\be\label{eq:S physical}
 \|S(t)\|_{L^1}
 =t^{b-3}\|\mathscr S(\log t)\|_{L^1}
 \leq C\, t^{b-3}(1+t^{2\lambda}).
\ee
By continuity of $\tau\mapsto\mathscr S(\tau)$ and of the dilations in $L^1(\R^2)$, we have $S\in C((0,1];L^1(\R^2))$, while $b>3$ implies $\|S(t)\|_{L^1}\to0$ as $t\to0^+$, so defining $S(0):=0$ gives
\[
 S\in C([0,1];L^1(\R^2)).
\]
Finally, applying the change of variables to the first velocity moment, we obtain
\be\label{eq:moment}
 \iint_{\R^2}|v|f_\pm(t,x,v)\,\di x\,\di v
 \leq t^{2b-3}\lf(
 \eps^{-1}\iint_{\R^2}|w|G(y,w)\,\di y\,\di w
 +\delta\eps\, t^\lambda
 \iint_{\R^2}|w||H_{\eps,\theta}(y,w)|\,\di y\,\di w
 \ri),
\ee
and both integrals are finite by \eqref{eq:Omega moments}, \eqref{eq:moments}, and \eqref{eq:pointwise}.

\step{6}{}
Since $\Omega_\eps(y,w)\leq\langle w\rangle^{-5}$,
\[
 \sup_{y\in\R}\int_\R\Omega_\eps(y,w)\,\di w<\infty.
\]
Combining the definition \eqref{eq:Gtheta} of $G$, the pointwise bound \eqref{eq:pointwise}, and the fixed spatial dilation, we deduce
\[
 \rho_{\overline F},\rho_h\in L^1(\R)\cap L^\infty(\R),
\]
so, by interpolation, both densities belong to $L^p(\R)$.
We then integrate \eqref{eq:fpm} in velocity to obtain
\be\label{eq:rhopm}
 \r_{f_\pm}(t,x)
 =t^{-2}\bigl[\rho_{\overline F}\pm\delta t^\lambda\rho_h\bigr](xt^{-b}),
\ee
and hence, using \eqref{eq:rho}, we get
\[
 \|\r_{f_\pm}(t)\|_{L^p}
 \leq t^{-2+b/p}
 \bigl(\|\rho_{\overline F}\|_{L^p}+\delta t^\lambda\|\rho_h\|_{L^p}\bigr)
 \leq C,
 \qquad 0<t\leq1,
\]
where the last inequality follows from $b\geq2p$ and $0<t\leq1$.

\step{7}{}
{To verify the weak formulation up to $t=0$, we first estimate $E_\pm f_\pm$.}
Using \eqref{eq:Epm} and $E[\rho_h]\in L^\infty$, together with \eqref{eq:mass} for the product, we obtain
\[
 \|E_\pm(t)\|_{L^\infty}\leq Ct^{b-2+\lambda},
 \qquad
 \|E_\pm(t)f_\pm(t)\|_{L^1}\leq Ct^{2b-4+\lambda},
\]
so, combining these bounds with \eqref{eq:mass}, \eqref{eq:moment}, and \eqref{eq:S physical}, we deduce that $f_\pm$, $vf_\pm$, $E_\pm f_\pm$, and $S$ are locally integrable up to $t=0$. Since $E_\pm$ does not depend on $v$, we write the force term as $E_\pm\pt_vf_\pm=\pt_v(E_\pm f_\pm)$. The ion density satisfies
\[
 \|\r_i(t)\|_{L^1}=t^{b-2}\|\rho_{\overline F}\|_{L^1},
\]
and is therefore also locally integrable up to $t=0$. We can now test the kinetic equation on $(\sigma,1)\times\R^2$ and let $\sigma\to0^+$, using the integrability estimates to pass to the limit in each space-time term and $\|f_\pm(\sigma)\|_{L^1}\to0$ to see that the time-boundary term vanishes. Together with the Poisson equation, this proves that both solutions satisfy \eqref{eq:vp} in the sense of distributions with the common zero initial datum.
\end{proof}

\appendix
\section{A uniqueness criterion and its temporal endpoint}
\label{sec:uniqueness}

We prove a weak-strong uniqueness criterion for \eqref{eq:vp} with a common additive source and a common prescribed ion density. The source may change sign, and the ion density may depend on time.
The proof uses the relation between the electric field and the spatial transport distance in one dimension, together with the second-order structure of the characteristics. Weak-strong estimates for the unforced one-dimensional equation were obtained in \cite{Hau13}, and they have been 
{adapted in \cite{HKI1} to study the quasineutral limit for the Vlasov-Poisson system with massless electrons.}
{We shall then use the solutions of Theorem~\ref{thm:main} to show that the time weight in this criterion is optimal.}

\subsection{The criterion and the role of the ion density}

We use the weighted space
\[
 L^1_1(\R^2)
 :=\left\{g:\ \|g\|_{L^1_1}
 :=\iint_{\R^2}(1+|x|+|v|)|g(x,v)|\,\di x\,\di v<\infty\right\}.
\]

\begin{theorem}[Weak-strong uniqueness]
\label{thm:uniqueness}
Let $T>0$, let $\rho_i\geq0$ belong to $L^1((0,T)\times\R)$, and let $S\in L^1((0,T)\times\R^2)$. Let $f_1,f_2\geq0$ be distributional solutions of \eqref{eq:vp}, with the same $S$, the same $\rho_i$, and the electric fields
\[
 E_j=E[\rho_{f_j}-\rho_i],\qquad j=1,2,
\]
normalized by \eqref{eq:field}. Suppose that
\be\label{eq:class}
 f_j\in C([0,T];L^1_1(\R^2)),\qquad f_1(0)=f_2(0),
\ee
and that one solution, say $f_1$, satisfies
\be\label{eq:uniqueness}
 E_1\in L^1(0,T;W^{1,\infty}(\R)),
 \qquad
 \int_0^T t\,\|\rho_{f_1}(t)\|_{L^\infty}\,\di t<\infty.
\ee
Then $f_1=f_2$ on $[0,T]$.
\end{theorem}

The density condition in \eqref{eq:uniqueness} also has a formulation in terms of the prescribed data.

\begin{corol}
\label{cor:ions}
{Assume all the hypotheses of Theorem~\ref{thm:uniqueness}, except the density condition in \eqref{eq:uniqueness}. Then that condition is equivalent to}
\be\label{eq:ion criterion}
 \int_0^T t\,\|\rho_i(t)\|_{L^\infty}\,\di t<\infty.
\ee
In particular, \eqref{eq:ion criterion} can replace the density condition in the uniqueness theorem.
\end{corol}

\begin{proof}
Since $\rho_{f_1}=\rho_i+\pt_xE_1$, we have, for almost every $t$,
\[
 \|\rho_{f_1}(t)\|_{L^\infty}
 \leq\|\rho_i(t)\|_{L^\infty}+\|\pt_xE_1(t)\|_{L^\infty},
 \qquad
 \|\rho_i(t)\|_{L^\infty}
 \leq\|\rho_{f_1}(t)\|_{L^\infty}+\|\pt_xE_1(t)\|_{L^\infty}.
\]
We multiply by $t$ and integrate, using $\int_0^Tt\|\pt_xE_1(t)\|_{L^\infty}\,\di t \leq T\|\pt_xE_1\|_{L^1_tL^\infty_x}<\infty$.
\end{proof}

{
\begin{remark}
For $A,B>0$, the transformation
\[
 \bal
 f_{A,B}(t,x,v)&:=AB\,f(At,Bx,(B/A)v),\\
 S_{A,B}(t,x,v)&:=A^2B\,S(At,Bx,(B/A)v),\\
 \rho_{i,A,B}(t,x)&:=A^2\rho_i(At,Bx),
 \qquad E_{A,B}(t,x):=\frac{A^2}{B}E(At,Bx)
 \eal
\]
preserves \eqref{eq:vp} and the normalization \eqref{eq:field}. In fact, $\rho_{f_{A,B}}(t,x)=A^2\rho_f(At,Bx)$, and every term of the kinetic equation gains the factor $A^2B$. Changing variables $s=At$, we obtain
\[
 \int_0^{T/A}t\|\rho_{f_{A,B}}(t)\|_{L^\infty}\,\di t
 =\int_0^T s\|\rho_f(s)\|_{L^\infty}\,\di s.
\]
The same identity holds for the ion density, so the weighted integrals in \eqref{eq:uniqueness} and \eqref{eq:ion criterion} are invariant under these Vlasov-Poisson dilations.
\end{remark}
}

\subsection{{Preliminary estimates}}

The first estimate identifies the quantity that we will control and allows us to use the density of the solution whose field is regular.

\begin{lemma}
\label{lem:field}
Let $\rho_1,\rho_2\geq0$ belong to $L^1(\R)$, have the same mass, and satisfy $\int_\R |x|(\rho_1+\rho_2)\,\di x<\infty$. Set $q=\rho_1-\rho_2$ and $e=E[q]$.
Then
\be\label{eq:field L1}
 e(x)=\int_{-\infty}^xq(y)\,\di y
     =-\int_x^\infty q(y)\,\di y,
 \qquad
 \|e\|_{L^1}\leq\int_\R |x||q(x)|\,\di x.
\ee
Moreover,
\be\label{eq:dual}
 \|e\|_{L^1}
 =\sup_{\substack{\varphi\in C^\infty(\R),\ \varphi'\in C_c^\infty(\R)\\
                         \|\varphi'\|_{L^\infty}\leq1}}
       \left|\int_\R\varphi(x)q(x)\,\di x\right|,
\ee
and
\be\label{eq:swap}
 \int_\R |e|\rho_2\,\di x=\int_\R |e|\rho_1\,\di x.
\ee
\end{lemma}

\begin{proof}
The two expressions for $e$ follow from \eqref{eq:field} and $\int q=0$. In particular, $e(\pm\infty)=0$ and $e'=q$. Integrating the first expression over $x<0$ and the second over $x>0$, and applying Fubini's theorem, we obtain
\[
 \int_{-\infty}^0|e(x)|\,\di x
 \leq\int_{-\infty}^0|y||q(y)|\,\di y,
 \qquad
 \int_0^\infty|e(x)|\,\di x
 \leq\int_0^\infty y|q(y)|\,\di y.
\]
This proves \eqref{eq:field L1}. Every test function in \eqref{eq:dual} is bounded, and integration by parts gives
\[
 \int_\R\varphi q\,\di x=-\int_\R\varphi'e\,\di x.
\]
The upper bound in \eqref{eq:dual} follows immediately. For the reverse bound, approximate $\sgn(e)$ by functions $\zeta_n\in C_c^\infty(\R)$ with $|\zeta_n|\leq1$ in $L^1(\R,|e|\,\di x)$, and take $\varphi_n(x)=\int_0^x\zeta_n(y)\,\di y$.

Finally, $e$ is bounded and absolutely continuous, so all the integrals in \eqref{eq:swap} are finite. Using $e'=\rho_1-\rho_2$, we find
\[
 \int_\R |e|(\rho_1-\rho_2)\,\di x
 =\int_\R |e|e'\,\di x
 =\frac12\bigl[e|e|\bigr]_{-\infty}^{+\infty}=0.
\]
\end{proof}

We next record the time factor obtained when a spatial test function is transported by a kinetic flow. Its proof also justifies the use of these tests for distributional solutions.

\begin{lemma}
\label{lem:test}
Let $E\in L^1(0,T;W^{1,\infty}(\R))$, let $g\in C([0,T];L^1_1(\R^2))$ with $g(0)=0$, and let $Q\in L^1((0,T)\times\R^2)$. Suppose that
\be\label{eq:difference}
 \pt_tg+v\pt_xg+\pt_v(Eg)=-\pt_vQ
\ee
in the sense of distributions. For every bounded $\varphi\in C^\infty(\R)$ with bounded derivative, and every $t\in[0,T]$, we have
\be\label{eq:test}
 \left|\iint_{\R^2}\varphi(x)g(t,x,v)\,\di x\,\di v\right|
 \leq C_T\|\varphi'\|_{L^\infty}
       \int_0^t(t-s)\|Q(s)\|_{L^1}\,\di s,
 \qquad
 C_T:=\exp\left(T\int_0^T\|\pt_xE(s)\|_{L^\infty}\,\di s\right).
\ee
\end{lemma}

\begin{proof}
We first regularize the field. Extend $E$ by zero outside $(0,T)$ and convolve with nonnegative smooth approximate identities in time and space. The resulting fields $E_n$ are smooth and satisfy
\be\label{eq:approx}
 \|E_n-E\|_{L^1_tL^\infty_x}\longrightarrow0,
 \qquad
 \int_0^T\|\pt_xE_n(s)\|_{L^\infty}\,\di s
 \leq\int_0^T\|\pt_xE(s)\|_{L^\infty}\,\di s.
\ee
For the spatial convolution, convergence follows from $\|E(s,\cdot-z)-E(s,\cdot)\|_{L^\infty} \leq |z|\|\pt_xE(s)\|_{L^\infty}$.
Time convolution converges in the Bochner space $L^1_tL^\infty_x$.

Let $(X_n(r;s,x,v),V_n(r;s,x,v))$ solve
\[
 \frac{\di X_n}{\di r}=V_n,\qquad
 \frac{\di V_n}{\di r}=E_n(r,X_n),\qquad
 (X_n,V_n)(s;s,x,v)=(x,v).
\]
The fields are bounded in space and have integrable spatial Lipschitz constants, so the flows are defined throughout $[0,T]$. We estimate their dependence on the initial velocity. Setting $m_n(r)=\pt_vX_n(r;s,x,v)$, we have
\[
 m_n(r)=r-s+\int_s^r(r-u)\pt_xE_n(u,X_n(u;s,x,v))m_n(u)\,\di u.
\]
For $r>s$, division by $r-s$ gives
\[
 \frac{|m_n(r)|}{r-s}
 \leq1+T\int_s^r\|\pt_xE_n(u)\|_{L^\infty}
                   \frac{|m_n(u)|}{u-s}\,\di u,
\]
where the quotient at $u=s$ is defined by its limit, equal to $1$. Gronwall's inequality and \eqref{eq:approx} imply
\be\label{eq:flow v}
 |\pt_vX_n(t;s,x,v)|\leq C_T(t-s).
\ee
The equation for $\pt_xX_n$ has initial values $1$ and $0$ in place of $0$ and $1$, and the same argument gives $|\pt_xX_n(t;s,x,v)|\leq C_T$.

For a fixed terminal time $t$, define the bounded test function
\[
 \psi_n(s,x,v)=\varphi(X_n(t;s,x,v)),\qquad 0\leq s\leq t.
\]
It satisfies
\[
 \pt_s\psi_n+v\pt_x\psi_n+E_n\pt_v\psi_n=0,
 \qquad \psi_n(t,x,v)=\varphi(x),
 \qquad
 |\pt_v\psi_n(s,x,v)|\leq C_T(t-s)\|\varphi'\|_{L^\infty}.
\]
To use it in \eqref{eq:difference}, take a smooth cutoff $\chi_R$ equal to $1$ on $[-R,R]$, supported in $[-2R,2R]$, with $|\chi_R'|\leq C/R$, and first test with $\psi_n(s,x,v)\chi_R(x)\chi_R(v)$. Time cutoffs can be removed using the continuity of $g$ in $L^1$ and the integrability of the fluxes. The terms in which a derivative falls on a spatial or velocity cutoff are bounded by
\[
 \frac{C\|\varphi\|_{L^\infty}}{R}
 \int_0^t\iint_{\R^2}
 \bigl(|v||g|+|E||g|+|Q|\bigr)\,\di x\,\di v\,\di s.
\]
This quantity tends to zero, since the first velocity moment of $g$ is uniformly bounded and $E\in L^1_tL^\infty_x$. Taking $R\to\infty$ for fixed $n$, we obtain
\be\label{eq:test identity}
 \iint_{\R^2}\varphi(x)g(t,x,v)\,\di x\,\di v
 =\int_0^t\iint_{\R^2}
       \bigl[Q+(E-E_n)g\bigr]\pt_v\psi_n\,\di x\,\di v\,\di s.
\ee
There is no initial term because $g(0)=0$. The contribution of the approximate field satisfies
\[
 \left|\int_0^t\iint_{\R^2}(E-E_n)g\pt_v\psi_n\,\di x\,\di v\,\di s\right|
 \leq C_TT\|\varphi'\|_{L^\infty}
       \sup_{0\leq s\leq T}\|g(s)\|_{L^1}
       \|E-E_n\|_{L^1_tL^\infty_x},
\]
which tends to zero. We bound the other term in \eqref{eq:test identity} by \eqref{eq:flow v} and take $n\to\infty$. This proves \eqref{eq:test} without requiring convergence of the derivatives of the approximate flows.
\end{proof}

\subsection{Proof of the uniqueness theorem}

\proofof{Theorem~\ref{thm:uniqueness}}
\step{1}{}
We write $\rho_j=\rho_{f_j}$ and set
\[
 g=f_1-f_2,\qquad q=\rho_1-\rho_2=\rho_g,
 \qquad e=E[q]=E_1-E_2.
\]
The source and the ion density cancel in the difference, and the kinetic equation becomes
\be\label{eq:g}
 \pt_tg+v\pt_xg+\pt_v(E_1g)=-\pt_v(ef_2),\qquad g(0)=0.
\ee
All the fluxes used below are integrable. Indeed, \eqref{eq:field} gives
\[
 \|E_j(t)\|_{L^\infty}
 \leq\frac12\bigl(\|f_j(t)\|_{L^1}+\|\rho_i(t)\|_{L^1}\bigr),
 \qquad
 \|e(t)\|_{L^\infty}\leq\frac12\|g(t)\|_{L^1}.
\]
The masses and first velocity moments of the solutions are uniformly bounded by \eqref{eq:class}. It follows that {$E_jf_j,E_1g,ef_2,vg\in L^1((0,T)\times\R^2)$}. In particular, the integrability of $ef_2$ does not use the density condition in \eqref{eq:uniqueness}.

Testing each equation with the cutoffs $\chi_R(x)\chi_R(v)$ used in Lemma~\ref{lem:test}, and then taking $R\to\infty$, we find
\[
 \iint_{\R^2}f_j(t)\,\di x\,\di v
 =\iint_{\R^2}f_j(0)\,\di x\,\di v
  +\int_0^t\iint_{\R^2}S(s)\,\di x\,\di v\,\di s.
\]
The transport cutoff errors tend to zero by the flux bounds just proved, while the source term converges because $S\in L^1$. {Since $f_1(0)=f_2(0)$, we obtain}
{
\[
 \iint_{\R^2}f_1(t)\,\di x\,\di v
 =\iint_{\R^2}f_2(t)\,\di x\,\di v,
 \qquad t\in[0,T],
\]
so Lemma~\ref{lem:field} applies to $\rho_1(t),\rho_2(t)$ and we can define}
\[
 A(t):=\|e(t)\|_{L^1_x}.
\]
The first spatial moment in \eqref{eq:class} and the bound in \eqref{eq:field L1} imply $e\in C([0,T];L^1(\R))$. To see the continuity directly, apply that bound to $q(t)-q(s)$, whose mass is zero, and use the weighted $L^1$ continuity of $g$.

\step{2}{}
Integrating \eqref{eq:g} in velocity, with a velocity cutoff followed by its removal, gives
\[
 \pt_tq+\pt_xj_g=0,\qquad
 j_g(t,x):=\int_\R v g(t,x,v)\,\di v.
\]
Here $j_g\in C([0,T];L^1(\R))$, and $j_g(0)=0$. Since $e'=q$ and $q(0)=0$, integration in time gives
\[
 \pt_x\left(e(t)+\int_0^t j_g(s)\,\di s\right)=0.
\]
The expression in parentheses belongs to $L^1(\R)$ and is spatially constant in the sense of distributions, so it is zero. Consequently,
\be\label{eq:e norm}
 e(t)=-\int_0^t j_g(s)\,\di s,
 \qquad
 A(t)\leq\int_0^t\|j_g(s)\|_{L^1}\,\di s.
\ee
In particular, $A(t)/t$ is bounded for $t>0$ and tends to zero as $t\to0^+$.

\step{3}{}
We apply Lemma~\ref{lem:test} to \eqref{eq:g}, with $E=E_1$ and $Q=ef_2$.
For every test function in \eqref{eq:dual}, it gives
\[
 \left|\int_\R\varphi(x)q(t,x)\,\di x\right|
 \leq C_T\int_0^t(t-s)\int_\R |e(s,x)|\rho_2(s,x)\,\di x\,\di s.
\]
Taking the supremum in \eqref{eq:dual}, and then using \eqref{eq:swap}, we obtain
\be\label{eq:Volterra}
 \bal
 A(t)
 &\leq C_T\int_0^t(t-s)\int_\R |e(s,x)|\rho_2(s,x)\,\di x\,\di s\\
 &=C_T\int_0^t(t-s)\int_\R |e(s,x)|\rho_1(s,x)\,\di x\,\di s\\
 &\leq C_T\int_0^t(t-s)\|\rho_1(s)\|_{L^\infty}A(s)\,\di s.
 \eal
\ee
The identity in the second line is what allows us to use the density of the same solution whose field transports the test functions.

Set $B(t)=A(t)/t$ for $t>0$ and $B(0)=0$. By \eqref{eq:e norm}, this is a bounded continuous function. Dividing \eqref{eq:Volterra} by $t$ and using $(t-s)/t\leq1$, we get
\[
 B(t)\leq C_T\int_0^t s\|\rho_1(s)\|_{L^\infty}B(s)\,\di s.
\]
The coefficient is integrable by \eqref{eq:uniqueness}. Gronwall's inequality implies $B=0$, and therefore $e=0$ on $[0,T]$.

\step{4}{}
It remains to recover $g=0$ from the equality of the fields. Equation \eqref{eq:g} now reduces to
\[
 \pt_tg+v\pt_xg+\pt_v(E_1g)=0,\qquad g(0)=0.
\]
Take any $\Phi\in C_c^\infty(\R^2)$ and repeat the approximation argument of Lemma~\ref{lem:test} with the terminal test
\[
 \Psi_n(s,x,v)=\Phi(X_n(t;s,x,v),V_n(t;s,x,v)).
\]
The full differential of the flow is bounded by $\exp(T+\int_0^T\|\pt_xE_1(s)\|_{L^\infty}\,\di s)$, uniformly in $n$.
There is no $Q$ term, and the remaining error is bounded by a constant times $\|E_1-E_{1,n}\|_{L^1_tL^\infty_x}\sup_s\|g(s)\|_{L^1}$, which tends to zero.
We conclude that $\iint\Phi g(t)=0$ for every such $\Phi$ and every $t\in[0,T]$.
Thus $f_1=f_2$, as claimed.
\end{proof}

\subsection{{Sharpness of the uniqueness criterion}}

{We now check that the solutions constructed in Section~\ref{sec:proof} satisfy all the assumptions of Theorem~\ref{thm:uniqueness} except the density condition in \eqref{eq:uniqueness}, and we compute the rate at which that condition fails.}
{
For a measurable function $k$ on $(0,1)$, we write
\[
 \|k\|_{L^{1,\infty}(0,1)}
 :=\sup_{\alpha>0}\alpha\,
       \bigl|\{t\in(0,1):|k(t)|>\alpha\}\bigr|.
\]
}

\begin{prop}
\label{prop:endpoint}
The solutions $f_\pm$ of Theorem~\ref{thm:main} satisfy
\[
 f_\pm\in C([0,1];L^1_1(\R^2)),\qquad f_\pm(0)=0,
 \qquad E_\pm\in L^1(0,1;W^{1,\infty}(\R)),
\]
their common ion density belongs to $L^1((0,1)\times\R)$, and their common source belongs to $C([0,1];L^1(\R^2))$. Also, there exists $m>0$ such that
\be\label{eq:limit}
 \lim_{t\to0^+}t^2\|\rho_{f_\pm}(t)\|_{L^\infty}=m,
 \qquad
 t^2\|\rho_i(t)\|_{L^\infty}=m\quad(0<t\leq1).
\ee
{
In particular,
\[
 t\|\rho_{f_\pm}(t)\|_{L^\infty},\quad t\|\rho_i(t)\|_{L^\infty}
 \in L^{1,\infty}(0,1)\setminus L^1(0,1),
\]
while
\[
 \int_0^1t^{1+\eta}\|\rho_{f_\pm}(t)\|_{L^\infty}\,\di t<\infty,
 \qquad
 \int_0^1t^{1+\eta}\|\rho_i(t)\|_{L^\infty}\,\di t=\frac m\eta,
 \qquad \eta>0,
\]
so the weight $t$ in \eqref{eq:uniqueness} or \eqref{eq:ion criterion} cannot be replaced by $t^{1+\eta}$ for any $\eta>0$, nor can $L^1$ integrability be replaced by weak $L^1$ integrability, while keeping the other hypotheses of the uniqueness theorem.

}
\end{prop}

\begin{proof}
Recall the profiles $\overline F,h$ used in \eqref{eq:fpm}. The compact spatial support of $n_\eps$, the first velocity moment of $\mu_{u,\gamma}$ in \eqref{eq:moments}, and \eqref{eq:Omega moments}, together with the pointwise estimate \eqref{eq:pointwise}, imply
\[
 \iint_{\R^2}(1+|\xi|+|w|)\bigl(\overline F(\xi,w)+|h(\xi,w)|\bigr)
       \,\di\xi\,\di w<\infty,
\]
and the fixed spatial dilation by $\eps$ preserves the finiteness of these moments. Since $a=-b-1$ and $c=b-1$, changing variables in \eqref{eq:fpm} gives
\be\label{eq:f moments}
 \bal
 \|f_\pm(t)\|_{L^1}&\leq Ct^{b-2},\\
 \iint_{\R^2}|x|f_\pm(t,x,v)\,\di x\,\di v&\leq Ct^{2b-2},\\
 \iint_{\R^2}|v|f_\pm(t,x,v)\,\di x\,\di v&\leq Ct^{2b-3}.
 \eal
\ee
All three quantities tend to zero because $b>3$, and at positive times the dilations are strongly continuous in $L^1_1$, which gives the claimed continuity in time.

The source was already estimated in \eqref{eq:S physical}. For the ion density, \eqref{eq:fpm} gives
\[
 \|\rho_i(t)\|_{L^1}=t^{b-2}\|\rho_{\overline F}\|_{L^1},
\]
which is integrable in time. Also, the field formula \eqref{eq:Epm} and its derivative give
\be\label{eq:E norm}
 \|E_\pm(t)\|_{L^\infty}
 =\delta t^{b-2+\lambda}\|E[\rho_h]\|_{L^\infty},\qquad
 \|\pt_xE_\pm(t)\|_{L^\infty}
 =\delta t^{\lambda-2}\|\rho_h\|_{L^\infty}.
\ee
The profile norms are finite by Section~\ref{sec:proof}, and \eqref{eq:parameters} gives $\lambda=\Lambda_0/\eps>b+1$, therefore both powers in \eqref{eq:E norm} are integrable at zero.

{
By Remark~\ref{rem:continuity} and the fixed spatial dilation, $\rho_h$ is continuous and vanishes at infinity, as does $E[\rho_h]\in H^1(\R)$. Since dilations of such functions are continuous in the uniform norm, \eqref{eq:Epm} gives continuity of $t\mapsto E_\pm(t)$ into $W^{1,\infty}(\R)$ for $t>0$, which \eqref{eq:E norm} extends to $t=0$ with value zero.
Thus $E_\pm\in L^1(0,1;W^{1,\infty}(\R))$ in the Bochner sense; in particular, their time mollifications converge in $L^1_tL^\infty_x$, as used in Lemma~\ref{lem:test}.

}

Set $m=\|\rho_{\overline F}\|_{L^\infty}>0$. By \eqref{eq:rhopm},
\[
 t^2\|\rho_{f_\pm}(t)\|_{L^\infty}
 =\|\rho_{\overline F}\pm\delta t^\lambda\rho_h\|_{L^\infty},
 \qquad
 \left|t^2\|\rho_{f_\pm}(t)\|_{L^\infty}-m\right|
 \leq\delta t^\lambda\|\rho_h\|_{L^\infty},
\]
while the ion identity in \eqref{eq:limit} follows from \eqref{eq:fpm}.
Thus there are constants $C_1,C_2>0$ and $t_0>0$ such that
\[
 \frac{C_1}{t}\leq t\|\rho_{f_\pm}(t)\|_{L^\infty}\leq\frac{C_2}{t}
 \quad(0<t<t_0),
 \qquad t\|\rho_i(t)\|_{L^\infty}=\frac{m}{t},
\]
which shows that the integrals in \eqref{eq:uniqueness} and \eqref{eq:ion criterion} diverge logarithmically. On the other hand, \eqref{eq:rhopm} gives $\|\rho_{f_\pm}(t)\|_{L^\infty}\leq C t^{-2}$ on the whole interval $(0,1]$,
{thus the superlevel sets of $t\|\rho_{f_\pm}(t)\|_{L^\infty}$ have measure at most $\min\{1,C/\alpha\}$ at level $\alpha>0$, which proves the weak $L^1$ bound; the same argument applies to $t\|\rho_i(t)\|_{L^\infty}=m/t$.}

Finally, for every $\eta>0$,
\[
 \int_0^1 t^{1+\eta}\|\rho_{f_\pm}(t)\|_{L^\infty}\,\di t
 \leq C\int_0^1t^{\eta-1}\,\di t<\infty,
 \qquad
 \int_0^1 t^{1+\eta}\|\rho_i(t)\|_{L^\infty}\,\di t=\frac{m}{\eta},
\]
which concludes the proof.
\end{proof}

\section*{Acknowledgments}
Part of this work was carried out while the author was a von Neumann Fellow at the Institute for Advanced Study during the 2024-2025 academic year.
The author acknowledges the support of the National Science Foundation through the grant No.~DMS-1926686.
This work was also supported by the Swiss National Science Foundation through the SNSF Starting Grant \emph{Challenges and Breakthroughs in the Mathematics of Plasmas}, grant no.~226018 (project TMSGI2\_226018). The author is grateful to the many colleagues with whom she has discussed the non-uniqueness problem for kinetic equations over the years.

\end{document}